\documentclass[11pt]{article}

\usepackage{amsmath,amssymb,amsthm,amsfonts,verbatim}
\usepackage{graphics,pst-plot,psfrag,epsfig}
\usepackage{graphicx}
\usepackage[all,2cell,dvips]{xy}

\title{Small dilatation pseudo-Anosovs and $3$--manifolds}

\author{Benson Farb, Christopher J. Leininger, and Dan Margalit
\thanks{The authors gratefully acknowledge support from the National
Science Foundation.}}

\theoremstyle{plain}
\newtheorem{theorem}{Theorem}[section]

\newtheorem{proposition}[theorem]{Proposition}

\newtheorem{lemma}[theorem]{Lemma}
\newtheorem{corollary}[theorem]{Corollary}

\newtheorem{question}[theorem]{Question}

\newtheorem*{theorem:pc}{Theorem \ref{theorem:principal congruence}}
\newtheorem*{theorem:jk}{Theorem \ref{theorem:johnson}}
\newtheorem*{theorem:br}{Theorem \ref{theorem:brunnian}}

\newcommand{\nc}{\newcommand}
\nc{\dmo}{\DeclareMathOperator}

\dmo{\SL}{SL} \nc{\C}{\mathcal{C}} \nc{\I}{\mathcal{I}} \nc{\K}{\mathcal{K}} \nc{\N}{\mathcal{N}} \nc{\R}{\mathbb{R}}
\nc{\Z}{\mathbb{Z}} \dmo{\PSL}{PSL} \dmo{\Teich}{Teich} \dmo{\spec}{spec} \nc{\gin}{i} \nc{\ga}{\Gamma} \dmo{\Out}{Out}
\dmo{\brun}{Brun} \dmo{\codeg}{codeg} \dmo{\intr}{int} \dmo{\vol}{vol}

\dmo\Sp{Sp}
\dmo\Mod{Mod}
\dmo\PMod{PMod}
\dmo\genus{genus}
\dmo\htop{h_{top}}
\dmo\CAT{CAT}
\dmo\M{{\cal M}}
\dmo\T{{\cal T}}
\dmo\sys{sys}
\dmo\out{out}

\nc{\margin}[1]{\marginpar{\tiny #1}}

\begin{document}
\maketitle
\vspace{-.25in}
\begin{abstract}
The main result of this paper is a universal finiteness theorem for the set of all small dilatation pseudo-Anosov homeomorphisms $\phi:S \to S$, ranging over all surfaces $S$.  More precisely, we consider pseudo-Anosovs $\phi:S \to S$ with $|\chi(S)|\log(\lambda(\phi))$ bounded above by some constant, and we prove that, after
puncturing the surfaces at the singular points of the stable foliations, the resulting set of mapping tori is finite.
Said differently, there is a finite set of fibered hyperbolic 3--manifolds so that all small dilatation
pseudo-Anosovs occur as the monodromy of a Dehn filling on one of the 3--manifolds in the finite list,
where the filling is on the boundary slope of a fiber.
\end{abstract}

\maketitle

%\pagewiselinenumbers

%---------------------
%---------------------
\section{Introduction}
%---------------------
%---------------------

Given a pseudo-Anosov homeomorphism $\phi$ of a surface $S$, let $\lambda(\phi)$ denote its dilatation.   For
any $P \geq 1$, we define
\[ \Psi_P = \left\{ \phi:S \to S \, \big{|} \, \chi(S)<0, \, \phi \mbox{ pseudo-Anosov, and } \lambda(\phi) \leq P^{\frac{1}{|\chi(S)|}} \right\}. \]
It follows from work of Penner \cite{Pe} that for $P$ sufficiently large, and for each closed surface $S_g$ of genus $g \geq 2$, there exists $\phi_g:S_g \to S_g$ so that
\[ \{ \phi_g : S_g \to S_g \}_{g=2}^\infty \subset \Psi_P. \]
We refer to $\Psi_P$ as the set of \textit{small dilatation pseudo-Anosov homeomorphisms}.

Given a pseudo-Anosov homeomorphism $\phi:S \to S$, let $S^\circ_\phi = S^\circ \subset S$ be the surface obtained by
removing the singularities of the stable and unstable foliations for $\phi$ and let $\phi|_{S^\circ}:S^\circ \to S^\circ$ denote the
restriction. The set of pseudo-Anosov homeomorphisms
\[ \Psi_P^\circ = \left\{ \phi|_{S^\circ}:S^\circ \to S^\circ \, |\, (\phi:S \to S) \in \Psi_P \right\} \]
is therefore also infinite.

The main discovery contained in this paper is a universal finiteness phenomenon for all small dilatation pseudo-Anosov
homeomorphisms: they are, in an explicit sense described below, all ``generated'' by a finite number of examples.  To
give a first statement, let $\T(\Psi_P^\circ)$ denote the homeomorphism classes of mapping tori of elements of $\Psi_P^\circ$.

\begin{theorem}
\label{theorem:main1} The set $\T(\Psi_P^\circ)$ is finite.
\end{theorem}

We will begin by putting Theorem \ref{theorem:main1} in context, and by explaining some of its restatements and corollaries.

\subsection{Dynamics and geometry of pseudo-Anosov homeomorphisms}

For a pseudo-Anosov homeomorphism $\phi:S \to S$,  the number $\lambda(\phi)$ gives a quantitative measure of several different
dynamical properties of $\phi$.  For example, given any fixed metric on the surface,  $\lambda(\phi)$ gives the growth rate of lengths of a
geodesic under iteration by $\phi$ \cite[Expos\'e 9, Proposition 19]{FLP}.  The number $\log(\lambda(\phi))$ gives the minimal topological
entropy of any homeomorphism in the isotopy class of $\phi$ \cite[Expos\'e 10, \S IV]{FLP}. Moreover, a pseudo-Anosov homeomorphism is
essentially the unique minimizer in its isotopy class---it is unique up to conjugacy by a
homeomorphism isotopic to the identity \cite[Expos\'e 12, Th\'eor\`eme III]{FLP}.

From a more global perspective, recall that the set of isotopy classes of orientation preserving homeomorphisms $\phi:S
\to S$ forms a group called the mapping class group, denoted $\Mod(S)$.  This group acts properly
discontinuously by isometries on the Teichm\"uller space $\Teich(S)$ with quotient the moduli space $\M(S)$ of Riemann
surfaces homeomorphic to $S$.  The closed geodesics in the orbifold $\M(S)$ correspond precisely to the conjugacy
classes of mapping classes represented by pseudo-Anosov homeomorphisms, and moreover, the length of a geodesic
associated to a pseudo-Anosov $\phi:S \to S$ is $\log(\lambda(\phi))$.  Thus, the length spectrum of $\M(S)$ is the set
\[ \spec(\Mod(S))=\{\log(\lambda(\phi)): \phi:S \to S \mbox{ is pseudo-Anosov}\} \subset (0,\infty).\]
Arnoux--Yoccoz \cite{AY} and Ivanov \cite{Iv} proved that $\spec(\Mod(S))$ is a closed discrete subset of $\R$. It
follows that $\spec(\Mod(S))$ has, for each $S$, a least element, which we shall denote by $L(S)$.  We can think of
$L(S)$ as the {\em systole} of $\M(S)$.

\subsection{Small dilatations}

Penner proved that there exists constants $0 < c_0 < c_1$ so that for all closed surfaces $S$ with $\chi(S) < 0$, one has
\begin{equation}
\label{eq:bounding dil} c_0 \leq L(S)   |\chi(S)|\leq c_1.
\end{equation}
The proof of the lower bound comes from a spectral estimate for Perron--Frobenius matrices, with $c_0 > \log(2)/6$ (see \cite{Pe} and \cite{Mc2}).   As such, this lower bound is valid for all surfaces $S$ with $\chi(S) < 0$,
including punctured surfaces.  The upper bound is proven by constructing pseudo-Anosov homeomorphisms $\phi_g:S_g \to
S_g$ on each closed surface of genus $g \geq 2$ so that $\lambda(\phi_g) \leq e^{c_1/(2g-2)}$; see also \cite{Ba}.

The best known upper bound for $\{L(S_g) |\chi(S_g)|\}$ is due to Hironaka--Kin \cite{HK} and independently Minakawa
\cite{Mk}, and is $2 \log(2 + \sqrt{3})$.  The situation for punctured surfaces is more mysterious.  There
is a constant $c_1$ so that the upper bound of \eqref{eq:bounding dil} holds for punctured spheres and punctured tori;
see \cite{HK,Ve,Ts}.  However, Tsai has shown that for a surface $S_{g,p}$ of fixed genus $g \geq 2$ and variable number
of punctures $p$, there is no upper bound $c_1$ for $L(S_{g,p}) |\chi(S_{g,p})|$, and in fact, this number grows like
$\log(p)$ as $p$ tends to infinity; see \cite{Ts}.

One construction for small dilatation pseudo-Anosov homeomorphisms is due to McMullen \cite{Mc2}.  The construction,
described in the next section, uses $3$--manifolds and is the motivation for our results.

\subsection{3--manifolds}

Given $\phi:S \to S$, the mapping torus $M = M_\phi$ is the total space of a fiber bundle $f:M \to S^1$, so that the
fiber is the surface $S$.   If $H^1(M;\R)$ has dimension at least $2$, then one can perturb the fibration to another
fibration $f':M \to S^1$ whose dual cohomology class is projectively close to the dual of $f:M \to S^1$; see \cite{Ti}.
In fact, work of Thurston \cite{Th2} implies that there is an open cone in $H^1(M;\R)$ with the property that every
integral class in this cone is dual to a fiber in a fibration.  Moreover, the absolute value of the euler
characteristic of these fibers extends to a linear function on this cone which we denote $|\chi(\cdot)|$.  

Fried \cite{Fr} proved that the dilatation of the monodromy extends to a continuous
function $\lambda(\cdot)$ on this cone, such that the reciprocal of its logarithm is homogeneous.  Therefore, the product $|\chi(\cdot)| \log(\lambda(\cdot))$ depends only on the projective class and varies continuously. McMullen \cite{Mc2} refined the analysis of the extension $\lambda(\cdot)$, producing a polynomial
invariant whose roots determine $\lambda(\cdot)$, and which provides a much richer structure.

In the course of his analysis, McMullen observed that for a fixed $3$--manifold $M = M_\phi$, if $f_k:M \to S^1$ is a
sequence of distinct fibrations of $M$ with fiber $S(k)$ and monodromy $\phi_k:S(k) \to S(k)$, for which the dual cohomology
classes converge projectively to the dual of the original fibration $f:M \to S^1$, then $|\chi(S(k))| \to \infty$, and
\[ |\chi(S(k))| \log(\lambda(\phi_k)) \to |\chi(S)| \log(\lambda(\phi)). \]
In particular, $|\chi(S(k))| \log(\lambda(\phi_k))$ is uniformly bounded, independently of $k$, and therefore
$\{\phi_k:S(k) \to S(k)\} \subset \Psi_P$ for some $P$.

There is a mild generalization of this construction obtained by replacing $\phi:S \to S$ with
$\phi|_{S^\circ}:S^\circ \to S^\circ$ (as defined above) and considering the resulting mapping torus $M^\circ = M_{\phi|_{S^\circ}}$.  McMullen's construction produces pseudo-Anosov homeomorphisms of punctured surfaces from the monodromies of various fibrations of $M^\circ \to S^1$.  By filling in any invariant set of punctures (for example, all of them), we obtain homeomorphisms that may or may not be pseudo-Anosov (after filling in the punctures, there may be a $1$--pronged singularity in the stable and unstable foliation, and this is not allowed for a pseudo-Anosov).  In any
case, we can say that all of the pseudo-Anosov homeomorphisms obtained in this way are \textit{generated by $\phi:S \to
S$}. Theorem \ref{theorem:main1} can be restated as follows.

\begin{corollary}
\label{corollary:fg}
Given $P > 1$, there exists a finite set of pseudo-Anosov homeomorphisms that generate all pseudo-Anosov
homeomorphisms $\phi \in \Psi_P$ in the sense above.
\end{corollary}

Corollary \ref{corollary:fg} can be viewed as a kind of converse of McMullen's construction.  In view of Corollary \ref{corollary:fg}, we have the following natural question.

\begin{question}
For any given $P$, can one explicitly find a finite set of pseudo-Anosov homeomorphisms that generate $\Psi_P$?
\end{question}

\subsection{Dehn filling}

Puncturing the surface at the singularities is necessary for the validity of Theorem \ref{theorem:main1}.  Indeed, the
set of all pseudo-Anosov homeomorphisms that occur as the monodromy for a fibration of a fixed $3$--manifold have a
uniform upper bound for the number of prongs at any singularity of the stable foliation.  On the other hand, Penner's
original construction produces a sequence of pseudo-Anosov homeomorphisms $\phi_g:S_g \to S_g$ in which the number of
prongs at a singularity tends to infinity with $g$.  So, the set $\T(\Psi_P)$ of homeomorphism classes of mapping tori of
elements of $\Psi_P$ is an infinite set for $P$ sufficiently large.

Removing the singularities from the stable and unstable foliations of $\phi$, then taking the mapping torus, is the
same as \textit{drilling out} the closed trajectories through the singular points of the suspension flow in $M_\phi$.
Thus, we can reconstruct $M = M_\phi$ from $M^\circ = M_{\phi|_{S^\circ}}$ by \textit{Dehn filling}; see \cite{Th1}.
The Dehn-filled solid tori in $M$ are regular neighborhoods of the closed trajectories through the singular points and
the \textit{Dehn filling slopes} are the (minimal transverse) intersections of $S$ with the boundaries of these
neighborhoods.  Back in $M^\circ$, the boundaries of the neighborhoods are tori that bound product neighborhoods of
the ends of $M^\circ$ homeomorphic to $[0,\infty) \times T^2$.  Here the filling slopes are described as the intersections of
$S^\circ$ with these tori and are called the \textit{boundary slopes of the fiber $S^\circ$}.

Moreover, since all of the manifolds in $\T(\Psi_P^\circ)$ are mapping tori for pseudo-Anosov homeomorphisms, they all admit
a complete finite volume hyperbolic metric by Thurston's Geometrization Theorem for Fibered $3$--manifolds (see, e.g.\
\cite{Mc1}, \cite{Ot}). As another corollary of Theorem \ref{theorem:main1}, we obtain the following.

\begin{corollary}
\label{corollary:maincor1} For each $P>1$ there exists $r\geq 1$ with the following property.  There are finitely many
complete, noncompact, hyperbolic $3$--manifolds $M_1, \ldots, M_r$ fibering over $S^1$, with the property that any $\phi
\in \Psi_P$ occurs as the monodromy of some bundle obtained by Dehn filling one of the $M_i$ along boundary slopes of a fiber.
\end{corollary}

\subsection{Volumes}

Because hyperbolic volume decreases after Dehn filling \cite{NZ,Th1}, as a corollary of Corollary
\ref{corollary:maincor1}, we have the following.

\begin{corollary} \label{C:bounded volume}
The set of hyperbolic volumes of $3$--manifolds in $\T(\Psi_P)$ is bounded by a constant depending only on $P$.
\end{corollary}

It follows that we can define a function $\mathcal V:\R\to\R$ by the formula:
\[ \mathcal V(\log P) = \sup \{\vol(M_\phi) \, | \, \phi \in \T(\Psi_P)\}.\]

Given a pseudo-Anosov $\phi:S \to S$, we have $\phi \in \Psi_P$ for $\log P = 
|\chi(S)|\log(\lambda(\phi))$, 
and so as a consequence of Corollary 1.5 we have the following.

\begin{corollary} For any pseudo-Anosov $\phi:S \to S$, we have
\[ \vol(M_\phi) \leq \mathcal V(|\chi(S)|\log(\lambda(\phi))). \]
\end{corollary}

Using the relationship between hyperbolic volume and simplicial volume \cite{Th2}, a careful analysis of our proof of Theorem~\ref{theorem:main1} shows that $\mathcal V$ is bounded above by an exponential function.

For a homeomorphism $\phi$ of $S$, let $\tau_{WP}(\phi)$ denote the 
translation length of $\phi$, thought of as an
isometry of $\Teich(S)$ with the Weil--Petersson metric. Brock \cite{Br} 
has proven that the volume of $M_\phi$ and
$\tau_{WP}(\phi)$ satisfy a bilipschitz relation (see also Agol \cite{Ag}), and in particular
\[\vol(M_\phi) \leq c \tau_{WP}(\phi).\]
Moreover, there is a relation between the Weil--Petersson
translation length and the Teichm\"uller translation length  $\tau_{\Teich}(\phi) = \log(\lambda(\phi))$ (see
\cite{Li}), which implies
\[ \tau_{WP}(\phi) \leq \sqrt{2\pi |\chi(S)|} \log(\lambda(\phi)). \]
However, Brock's constant $c=c(S)$ depends on the surface $S$, and moreover $c(S) \geq |\chi(S)|$ when $|\chi(S)|$ is sufficiently large.
In particular, Corollary \ref{C:bounded volume} does not follow from these estimates.  See also \cite{KKT} for a discussion
relating volume to dilatation for a fixed surface.

\subsection{Minimizers}

Very little is known about the actual values of $L(S)$.   It is easy to prove $L(S_1)=\log(\frac{3+\sqrt{5}}{2})$. The
number $L(S)$ is also known when $|\chi(S)|$ is relatively small; see \cite{CH,HS,HK,So,SKL,Zh,LT}. However, the exact
value is not known for any surface of genus greater than $2$; see \cite{LT} for some partial results for surfaces of
genus less than $9$.  In spite of the fact that elements in the set $\{L(S)\}$ seem very difficult to determine, the
following corollary shows that infinitely many of these numbers are generated by a single example.

\begin{corollary} \label{C:magic?}
There exists a complete, noncompact, finite volume, hyperbolic $3$--manifold $M$ with the following property: there
exist Dehn fillings of $M$ giving an infinite sequence of fiberings over $S^1$, with closed fibers $S_{g_i}$ of genus
$g_i\geq 2$ with $g_i\to\infty$, and with monodromy $\phi_i$,  so that
$$L(S_{g_i})=\log (\lambda(\phi_i)).$$
\end{corollary}

We do not know of an explicit example of a hyperbolic $3$--manifold as in Corollary \ref{C:magic?}; the work of \cite{KKT,KT,Ve}
gives one candidate.

We now ask if the Dehn filling is necessary in the statement of Corollary~\ref{C:magic?}:

\begin{question}
Does there exist a single $3$--manifold that contains infinitely many minimizers?  That is, does there exist a
hyperbolic $3$--manifold $M$ that fibers over the circle in infinitely many different ways $\phi_k:M \to S^1$, so that
the monodromies $f_k:S_{g_k,p_k} \to S_{g_k,p_k}$ satisfy $\log(\lambda(\phi_k)) = L(S_{g_k,p_k})$?
\end{question}

\subsection{Outline of the proof}

To explain the motivation for the proof, we again consider McMullen's construction.  In a fixed fibered $3$--manifold
$M \to S^1$, Oertel proved that all of the fibers $S(k)$ for which the dual cohomology classes lie in the open cone
described above are carried by a finite number of \textit{branched surfaces} transverse to the suspension flow.  A
branched surface is a $2$--dimensional analogue in a $3$--manifold of a train track on a surface; see \cite{Oe}.

Basically, an infinite sequence of fibers in $M$ whose dual cohomology classes are projectively converging are building
up larger and larger ``product regions'' because they all live in a single $3$--manifold.  These product regions can be
collapsed down, and the images of the fibers under this collapse define a finite number of branched surfaces.

Our proof follows this idea by trying to find ``large product regions'' in the mapping torus that can be collapsed, so
that the fiber projects onto a branched surface of uniformly bounded complexity.  In our case, we do not know that we
are in a fixed $3$--manifold (indeed, we may not be), so the product regions must come from a different source than in
the single $3$--manifold setting; this is where the small dilatation assumption is used.  Moreover, the collapsing
construction we describe does not in general produce a branched surface in a $3$--manifold.  However, the
failure occurs only along the closed trajectories through the singular points, and after removing the singularities,
the result of collapsing is indeed a $3$--manifold.

While this is only a heuristic (for example, the reader never actually needs to know what a branched surface
is), it is helpful to keep in mind while reading the proof, which we now outline.  First we replace all punctures by
marked points, so $S$ is a closed surface with marked points. Let $M = M_\phi$ be the mapping torus, which is now a
compact $3$--manifold.  We associate to any small dilatation pseudo-Anosov homeomorphism $\phi:S\to S$ a
Markov partition ${\cal R}$ with a ``small'' number of rectangles [Section \ref{S:markov partition}].

\medskip
{\em Step 1:} We prove that all but a uniformly bounded number of the rectangles of $\mathcal R$ are essentially
permuted [Lemma \ref{L:finite nonhomeo}].  This follows from the relationship with Perron--Frobenius matrices and their adjacency graphs together with an application of a result of Ham and Song \cite{HS} [Lemma \ref{L:matrix lemma}].  Moreover, we prove that any rectangle meets a uniformly bounded number of other rectangles
along its sides [Lemma \ref{L:bounded complexity 1}].  These rectangles, and their images, are used to construct a cell
structure $Y$ on $S$ [Section \ref{S:cell structures}].

\smallskip
{\em Step 2:} When a rectangle $R\in{\cal R}$ and all of the rectangles ``adjacent'' to it are taken homeomorphically
onto other rectangles of the Markov partition by both $\phi$ and $\phi^{-1}$, then we declare $R$ and $\phi(R)$ to be
``$Y$--equivalent'' [Section \ref{S:equiv rel}].  This generates an equivalence relation on rectangles with a uniformly bounded
number of equivalence classes [Corollary \ref{C:relation 1 prop}]. Moreover, if two different rectangles are equivalent
by a power of $\phi$, then that power of $\phi$ is cellular on those rectangles with respect to the cell structure $Y$ on $S$ [Proposition
\ref{P:sim1 and Y}].

\smallskip
{\em Step 3:  }In the mapping torus $M_\phi$, the suspension flow $\phi_t$ applied to the rectangle for $0 \leq t \leq
1$ defines a ``box'' in the 3--manifold, and applying this to all rectangles in ${\cal R}$ produces a decomposition of
$M$ into boxes which can be turned into a cell structure $\hat Y$ on $M$ so that $S \subset M$ with its cell structure $Y$ is a subcomplex [Section \ref{S:3-mfd}].

\smallskip
{\em Step 4: }If a rectangle $R$ is $Y$--equivalent to $\phi(R)$, then we call the box constructed from $R$ a ``filled
box''.  The filled boxes stack end-to-end into ``prisms'' [Section \ref{S:quotient I}]. Each of these prisms admits a
product structure of an interval times a rectangle, via the suspension flow.  We construct a quotient $N$ of $M$ by
collapsing out the flow direction in each prism.  The resulting $3$--complex $N$ has a uniformly bounded number
of cells in each dimension 0, 1, 2, 3, and with a uniform bound on the complexity of the attaching maps [Proposition
\ref{P:hat W bounded}].  There are finitely many such $3$--complexes [Proposition \ref{P:bounded
complexity finite}]. The uniform boundedness is a result of careful analysis of the cell structures $Y$ and $\hat Y$ on $S$ and $M$, respectively.

\smallskip
{\em Step 5: }The flow lines through the singular and marked points are closed trajectories in $M$ that become 1--subcomplexes in
the quotient $N$.   We remove all of these closed trajectories from $M$ to
produce $M^\circ \in \T(\Psi_P^\circ)$ and we remove their image $1$--subcomplexes from $N$ to produce $N^\circ$.  We then prove that the associated quotient $M^\circ \to N^\circ$ is
homotopic to a homeomorphism [Theorem \ref{T:quotient htpc homeo}].  Thus, $N^\circ \cong M^\circ$ is obtained from a finite
list of $3$--complexes by removing a finite $1$--subcomplex, which completes the proof.

\bigskip
\noindent \textbf{Acknowledgements.} We would like to thank Mladen Bestvina, Jeff Brock, Dick Canary, and Nathan
Dunfield for helpful conversations.

%%%%%%%%%%%%%%%%%%%%%%%%%%%%%%%%%%%%%%%%%%%%%%%%%%%%%%%%%%%%%%%%%%%%%%%%%%%%%%%%%%%%%%%%%
%%%%%%%%%%%%%%%%%%%%%%%%%%%%%%%%%%%%%%%%%%%%%%%%%%%%%%%%%%%%%%%%%%%%%%%%%%%%%%%%%%%%%%%%%
%%%%%%%%%%%%%%%%%%%%%%%%%%%%%%%%%%%%%%%%%%%%%%%%%%%

%%%%%%%%%%%%%%%%%%%%%%%%%%%%%%%%%%%%%%%%
\section{Finiteness for CW--complexes}
%%%%%%%%%%%%%%%%%%%%%%%%%%%%%%%%%%%%%%%%
\label{S:finiteness}

Our goal is to prove that the set of homeomorphism types of the $3$--manifolds in $\T(\Psi_P^\circ)$ is finite.  We will accomplish this by finding a finite
list of compact $3$--dimensional CW--complexes $\overline{\T(\Psi_P)}$ with the property that any $3$--manifold in $\T(\Psi_P^\circ)$
is obtained by removing a finite subcomplex of the $1$--skeleton from one of the $3$--complexes in $\overline{\T(\Psi_P)}$.

To prove that the set of $3$--complexes in $\overline{\T(\Psi_P)}$ is finite, we will first find a constant $K = K(P)$ so that
each complex in $\overline{\T(\Psi_P)}$ is built from at most $K$ cells in each dimension.  This alone is not enough to
conclude finiteness. For example, one can construct infinitely many homeomorphism types of $2$--complexes using one
$0$--cell, one $1$--cell, and one $2$--cell. To conclude finiteness, we must therefore also impose a bound to the
complexity of our attaching maps.

For an integer $K \geq 0$, we define the notion of \textit{$K$--bounded complexity} for an $n$--cell of a CW--complex
$X$ in the following inductive manner.  All $0$--cells have $K$--bounded complexity for all $K \geq 0$.  Having defined
$K$--bounded complexity for $(n-1)$--cells, we say that an $n$--cell has $K$--bounded complexity if the attaching map
$\partial \mathbb D^n \to X^{(n-1)}$ has the following properties:
\begin{enumerate}
\item The domain $\partial \mathbb D^n$ can be given the structure of an $(n-1)$--complex with at most $K$ cells in each dimension so that
each cell has $K$--bounded complexity, and
\item The attaching map is a homeomorphism on the interior of each cell.
\end{enumerate}
Observe that a $1$--cell has $K$--bounded complexity for all $K \geq 2$.  A $2$--cell has $K$--bounded complexity if
the boundary is subdivided into at most $K$ arcs, and the attaching map sends the interior of these arcs
homeomorphically onto the interiors of $1$--cells in $X^{(1)}$.

A finite cell complex has \textit{$K$--bounded complexity} if each cell has $K$--bounded complexity and there are at
most $K$ cells in each dimension.  Such a cell complex is special case of a \textit{combinatorial complex}; see
\cite[I.8 Appendix]{BrH}.

\begin{proposition} \label{P:bounded complexity finite}
Fix integers $K,n \geq 0$.  The set of CW--homeomorphism types of $n$--complexes with $K$--bounded complexity is
finite.
\end{proposition}
\begin{proof}
Let $X$ be an $n$--complex with $K$--bounded complexity.  We may subdivide each cell of $X$ in order to obtain a
complex $X'$ where each cell is a simplex and each attaching map is a  homeomorphism on the interior of each cell; in the language of \cite{Ha}, such
a complex is called a $\Delta$--complex. What is more, we may choose $X'$ so that the number of cells is bounded above
by a constant that only depends on $K$ and $n$.  One way to construct $X'$ from $X$ is to subdivide inductively by
skeleta as follows.

Since any graph is a $\Delta$--complex, there is nothing to do for the 1-skeleton.  To make the 2--skeleton
into a $\Delta$--complex, it suffices to add one vertex to the interior of each 2--cell, and an edge connecting the new
vertex to each vertex of the boundary of the 2--cell; in other words, we cone off the boundary of each 2--cell. We then
cone off the $3$--cells, then $4$--cells, and so on, until at last we cone off all the $n$--cells.  The number of
simplices in any dimension is bounded by a function of $K$ and the dimension, so in particular, there is a uniformly
bounded number of simplices, depending only on $K$ and $n$.

The second barycentric subdivision $X''$ of $X'$ is a simplicial complex (this is true for any $\Delta$--complex), and
the number of vertices of $X''$ is bounded above by some constant $K'$ that only depends on $K$ and $n$.  Observe that $X$ is
homeomorphic to $X''$.  In fact, the CW--homeomorphism type of $X$ is determined by the simplicial isomorphism type
of $X''$, up to finite ambiguity:  there are
only finitely many CW--homeomorphism types that subdivide to the simplicial
isomorphism type of $X''$.  We also note that $X''$ is determined up to simplicial isomorphism by specifying
which subsets of the $K'$ vertices span a simplex.  Therefore, there are only finitely many simplicial isomorphism
types of $X''$ and so only finitely many CW--homeomorphisms types of $X$.
\end{proof}

\bigskip

%%%%%%%%%%%%%%%%%%%%%%%%%%%%%%%%%%%%%%%%%%%%%%%%%%%%%%%%%%%%%%%%%%%%%%%%%%%%%%%%%%%%%%%%%
%%%%%%%%%%%%%%%%%%%%%%%%%%%%%%%%%%%%%%%%%%%%%%%%%%%%%%%%%%%%%%%%%%%%%%%%%%%%%%%%%%%%%%%%%
%%%%%%%%%%%%%%%%%%%%%%%%%%%%%%%%%%%%%%%%%%%%%%%%%%%%%%%%%%%%%%%%%%%%%%%%%%%%%%%%%%%%%%%%%

%%%%%%%%%%%%%%%%%%%%%%%%%%%%%%%%%%%%%%%%%%%%%%%%%%
%%%%%%%%%%%%%%%%%%%%%%%%%%%%%%%%%%%%%%%%%%%%%%%%%%
\section{The adjacency graph for a Perron--Frobenius matrix} \label{S:spectral bound}
%%%%%%%%%%%%%%%%%%%%%%%%%%%%%%%%%%%%%%%%%%%%%%%%%%
%%%%%%%%%%%%%%%%%%%%%%%%%%%%%%%%%%%%%%%%%%%%%%%%%%

Our proof of Theorem \ref{theorem:main1} requires a few facts about Perron--Frobenius matrices and their associated
adjacency graphs.  After recalling these definitions, we give a result of Ham and Song \cite{HS} that bounds the complexity of an adjacency graph in terms of the spectral radius of the associated matrix.

\medskip

\begin{figure}[h!]
\centerline{}\centerline{}
\begin{center}
\ \psfig{file=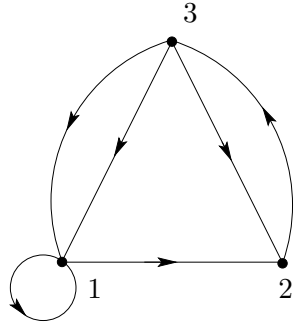,height=1.5truein} \caption{The adjacency graph for the $3 \times 3$ matrix $A$. }  \label{F:adjacency}
\end{center}
  \setlength{\unitlength}{1in}
  \begin{picture}(0,0)(0,0)
    \put(2.7,.8){$1$}
    \put(3.7,.8){$2$}
    \put(3.2,2.25){$3$}
  \end{picture}
\end{figure}

Let $A$ be an $n\times n$ {\em Perron--Frobenius matrix}, that is, $A$ has nonnegative integer entries $a_{ij}\geq 0$
and there exists a positive integer $k$ so that each entry of $A^k$ is positive. Associated to any Perron--Frobenius
matrix $A$ is its {\em adjacency graph} $\Gamma_A$. This is a directed graph with $n$ vertices labeled $\{1,\ldots
,n\}$ and $a_{ij}$ directed edges from vertex $i$ to vertex $j$. For example, Figure \ref{F:adjacency} shows the
adjacency graph for the matrix:
\[ A = \left( \begin{array}{ccc} 1 & 1 & 0\\ 0 & 0 & 1\\ 2 & 1 & 0\\
\end{array} \right) \]
The $(i,j)$--entry of $A^m$ is the number of distinct combinatorial directed paths of length $m$ from vertex $i$ to vertex $j$ of $\Gamma_A$.  In particular, it follows that if $A$ is a Perron--Frobenius matrix, then there is a directed path from any vertex to any other vertex.

\begin{remark} A Perron--Frobenius matrix is sometimes defined as a matrix $A$ with nonnegative integer entries with the property that, for each pair $(i,j)$, there is a $k$---depending on $(i,j)$---so that the $(i,j)$--entry of $A^k$ is positive.  As a consequence of our stronger definition, every positive power of a Perron--Frobenius matrix is also Perron--Frobenius.
\end{remark}

\bigskip

The Perron--Frobenius Theorem (see e.g.~\cite[\S III.2]{Ga}) tells us that $A$ has a positive eigenvalue $\lambda(A)$, called the
\textit{spectral radius}, which is strictly greater than the modulus of every other eigenvalue and which has a
$1$--dimensional eigenspace spanned by a vector with all positive entries.  Moreover, one can bound $\lambda(A)$ from
above and below by the maximal and minimal row sums, respectively:
\begin{equation}
\label{eq:PF1} \min_i\sum_ja_{ij}\leq \lambda(A)\leq \max_i\sum_ja_{ij}.
\end{equation}

For a directed graph $\Gamma$ and a vertex $v$ of $\Gamma$, let $\deg_{\textnormal{out}}(v)$ and $\deg_{\textnormal{in}}(v)$ denote the number
of edges beginning and ending at $v$, respectively.  Since each edge has exactly one initial endpoint and one terminal
endpoint, it follows that the number of edges of $\Gamma$ is precisely
\[ \sum_{v \in \Gamma^{(0)}} \deg_{\textnormal{out}}(v) = \sum_{v \in \Gamma^{(0)}} \deg_{\textnormal{in}}(v). \]

The following fact is due to Ham--Song \cite[Lemma 3.1]{HS}

\begin{lemma} \label{L:matrix lemma}
Let $A$ be an $n\times n$ Perron--Frobenius matrix with adjacency graph $\Gamma_A$.  We have
\begin{equation}
 \label{eq:deg1}
 1+\sum_{v \in \Gamma_A^{(0)}} (\deg_{\textnormal{out}}(v) - 1) = 1+\sum_{v \in \Gamma_A^{(0)}} (\deg_{\textnormal{in}}(v) - 1)  \leq \lambda(A)^n.
\end{equation}
\end{lemma}

Both sums in the statement of Lemma~\ref{L:matrix lemma} are equal to $-\chi(\Gamma_A)$.  In particular, Lemma~\ref{L:matrix lemma} bounds the number of homeomorphism types of graphs $\Gamma_A$ in terms of $\lambda(A)$, since $\Gamma_A$ cannot have any valence
$1$ vertices.

As Lemma~\ref{L:matrix lemma} is central to our paper, we give the proof due to Ham and Song.

\begin{proof}

Let $T$ be a spanning tree for $\Gamma_A$.  Since $T^{(0)}=\Gamma_A^{(0)}$, we have that $T$ has exactly $n$ vertices.  Since $T$ is a tree, we have $\chi(T) = 1$, and so $T$ has $n-1$ edges.

We claim that the number of directed paths of length $n$ starting from a vertex $v$ of $\Gamma_A$ is greater than or equal to the number of edges of $\Gamma_A$ not contained in $T$.  Indeed, each path of length $n$ must leave $T$, and so there is a surjective set map from the set of directed paths of length $n$ based at $v$ to the set of edges of $\Gamma_A$ not contained in $T$: for each such path, take the first edge in the path that is not an edge of $T$.

Let $v_0$ be the vertex of $\Gamma_A$ corresponding to the row of $A^n$ that has the smallest row sum.  We have:
\begin{eqnarray*}
\lambda(A)^n &\geq& \mbox{smallest row sum for } A^n \\
&=& \mbox{number of directed paths of length } n \mbox{ starting from } v_0 \\
&\geq& \mbox{number of edges of } \Gamma_A \mbox{ not contained in } T \\
&=& \left(\sum_{v \in \Gamma_A^{(0)}} \deg_{\textnormal{out}}(v)\right) - (n-1) \\
&=& 1+\sum_{v \in \Gamma_A^{(0)}} (\deg_{\textnormal{out}}(v)-1).
\end{eqnarray*}
Since $\sum \deg_{\textnormal{out}}(v) = \sum \deg_{\textnormal{in}}(v)$, we are done.
\end{proof}

%%%%%%%%%%%%%%%%%%%%%%%%%%%%%%%%%%%%%%%%%%%%%%%%%%%%%%%%%%%%%%%%%%%%%%%%%%%%%%%%%%%%%%%%%
%%%%%%%%%%%%%%%%%%%%%%%%%%%%%%%%%%%%%%%%%%%%%%%%%%%%%%%%%%%%%%%%%%%%%%%%%%%%%%%%%%%%%%%%%
%%%%%%%%%%%%%%%%%%%%%%%%%%%%%%%%%%%%%%%%%%%%%%%%%%%%%%%%%%%%%%%%%%%%%%%%%%%%%%%%%%%%%%%%%
%%%%%%%%%%%%%%%%%%%%%%%%%%%%%%%%%%%%%%%%%%%%%%%%%%%%%%%%%%%%%%%%%%%%%%%%%%%%%%%%%%%%%%%%%
%%%%%%%%%%%%%%%%%%%%%%%%%%%%%%%%%%%%%%%%%%%%%%%%%%%%%%%%%%%%%%%%%%%%%%%%%%%%%%%%%%%%%%%%%

%%%%%%%%%%%%%%%%%%%%%%%%%%%%%%%%%%%%%%%%%%%%%%%%%%%%%%
%%%%%%%%%%%%%%%%%%%%%%%%%%%%%%%%%%%%%%%%%%%%%%%%%%%%%%
\section{Pseudo-Anosov homeomorphisms and Markov partitions} \label{S:markov partition}
%%%%%%%%%%%%%%%%%%%%%%%%%%%%%%%%%%%%%%%%%%%%%%%%%%%%%%
%%%%%%%%%%%%%%%%%%%%%%%%%%%%%%%%%%%%%%%%%%%%%%%%%%%%%%

Let $S$ be a surface of genus $g$ with $p$ marked points (marked points are more convenient for us than punctures in
what follows).  We still write $\chi(S) = 2- 2g-p$ and assume $\chi(S) < 0$.

%%%%%%%%%%%%%%%%%%%%%%%%%%%%%%%%%%%%%%%%%%%%%%%
\subsection{Pseudo-Anosov homeomorphisms}
%%%%%%%%%%%%%%%%%%%%%%%%%%%%%%%%%%%%%%%%%%%%%%%

First recall that one can describe a complex structure and integrable holomorphic quadratic differential $q$ on $S$ by
a Euclidean cone metric with the following properties:
\begin{enumerate}
\item Each cone angle has the form $k\pi$ for some $k \in \mathbb Z_+$, with $k \geq 2$ at any unmarked point $z \in
S$.
\item There is an orthogonal pair of singular foliations $\mathcal F_h$ and $\mathcal F_v$ on $S$, called the \textit{horizontal} and \textit{vertical foliations}, respectively, with all singularities at the cone points, and with all leaves geodesic.
\end{enumerate}
Such a metric has an atlas of \textit{preferred coordinates} on the complement of the singularities, which are local
isometries to $\R^2$ and for which the vertical and horizontal foliations are sent to the vertical and horizontal
foliations of $\R^2$.

\bigskip

A homeomorphism $\phi:S \to S$ is pseudo-Anosov if and only if there exists a complex structure on $S$, a quadratic differential $q$ on $S$, and a $\lambda > 1$, so that in any preferred coordinate chart for $q$, the map $\phi$ is given by
\[ (x,y) \mapsto (\lambda x + c,\frac{y}{\lambda} + c') \]
for some $c,c' \in \R$.  In particular, observe that $\phi$ preserves $\mathcal F_h$ and $\mathcal F_v$.  The horizontal foliation is called the \textit{stable foliation}, and the leaves are all stretched by $\lambda$;
the vertical foliation is called the \textit{unstable foliation}, and the leaves of this foliation are contracted.  The number
$\lambda$ is nothing other than the dilatation $\lambda = \lambda(\phi)$.

%%%%%%%%%%%%%%%%%%%%%%%%%%%%%%
\subsection{Markov partitions}
%%%%%%%%%%%%%%%%%%%%%%%%%%%%%%

The main technical tool we will use for studying pseudo-Anosov homeomorphisms is the theory of Markov partitions, which we now
describe.  Let $\phi$ be a pseudo-Anosov homeomorphism, and let $q$ be a quadratic differential as in our description of a pseudo-Anosov homeomorphism.
A \textit{rectangle} (for $q$) is an immersion
\[ \rho:[0,1] \times [0,1] \to S\]
that is affine with respect to the preferred coordinates, and that satisfies the following properties:
\begin{enumerate}
\item $\rho$ maps the interior injectively onto an open set in $S$,
\item $\rho(\{x\} \times [0,1])$ is contained in a leaf of $\mathcal F_v$ and $\rho([0,1] \times \{y\})$ is contained in a leaf of $\mathcal
F_h$, for all $x \in (0,1)$ and $y \in (0,1)$, and
\item $\rho(\partial ([0,1] \times [0,1]))$ is a union of arcs of leaves and singularities of $\mathcal F_v$ and $\mathcal F_h$.
\end{enumerate}
As is common practice, we abuse notation by confusing a rectangle and its image $R = \rho([0,1] \times [0,1])$.  The
\textit{interior} of a rectangle $\intr(R)$ is the $\rho$--image of its interior.

For any rectangle $R$ define
$$\partial_v R=\rho(\{0,1\}\times [0,1])\ \ \mbox{and}\ \  \partial_h R=\rho([0,1]\times \{0,1\}).$$
We thus have $\partial R=\partial_v R \cup \partial_h R$.

\medskip

A {\em Markov partition}\footnote{What we call a ``Markov partition'' is sometimes called a ``pre-Markov partition'' in
the literature; see, e.g., Expos\'{e} 9, Section 5 of \cite{FLP}.} for a pseudo-Anosov homeomorphism $\phi:S \to S$ is
a finite set of rectangles $\mathcal R = \{R_i\}_{i=1}^n$ satisfying the following:
\begin{enumerate}
\item $\intr(R_i)\cap \intr(R_j)=\emptyset$ for $i \neq j$,
\item $S=R_1\cup\cdots\cup R_n$,
\item $\phi(\cup \partial_v R_i) \subseteq \cup\partial_v R_i$,
\item $\phi^{-1}(\cup \partial_h R_i)\subseteq \cup \partial_h R_i$, and
\item each marked point of $S$ lies on the boundary of some $R_i$.
\end{enumerate}

\begin{figure}[htb]
\centerline{}\centerline{}
\begin{center}
\ \psfig{file=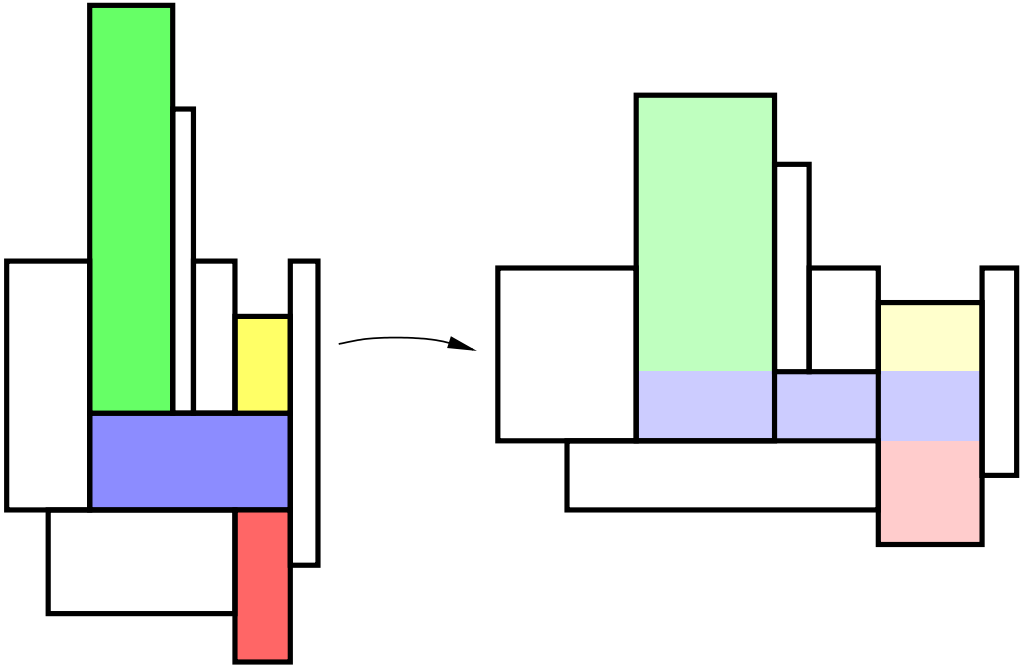,height=3.5truein} \caption{A rectangle in a Markov partition is stretched horizontally and compressed vertically, then mapped over other rectangles, sending vertical sides to vertical sides.  In this local picture there are 9 rectangles in one part of the surface being mapped over 8 rectangles in some other part of the surface.  The colored rectangles are the ones that are mixed, while the other rectangles are unmixed.}  \label{F:markovpic1}
\end{center}
\end{figure}

We say that a Markov partition  $\mathcal R = \{ R_i\}_{i=1}^n$ for $\phi$ is \emph{small} if
\[ |\mathcal R| \leq 9|\chi(S)|. \]
The following lemma appears in the work of Bestvina--Handel; see  \S 3.4, \S 4.4, and \S 5 of \cite{BH}.

\begin{lemma}
Suppose $\chi(S)<0$.  Every pseudo-Anosov homeomorphism of $S$ admits a small Markov partition.
\end{lemma}

%%%%%%%%%%%%%%%%%%%%%%%%%%%%%%%%%%%%%%%%%%%%%%%%%%%%%%%
\subsection{The adjacency graph of a Markov partition}
%%%%%%%%%%%%%%%%%%%%%%%%%%%%%%%%%%%%%%%%%%%%%%%%%%%%%%%

Given a Markov partition $\mathcal R = \{R_i\}_{i=1}^n$ for $\phi:S \to S$, there is an associated nonnegative $n
\times n$ integral matrix $A = A(\phi,\mathcal R)$, called the \textit{transition matrix} for $(\phi,\mathcal R)$, whose
$(i,j)$--entry is
$$a_{ij}=\left| \phi^{-1}(\intr(R_j))\cap R_i\right|,$$
where the absolute value sign denotes the number of components.  That is, $A$ records how many times the $i^{th}$
rectangle ``maps over'' the $j^{th}$ rectangle after applying $\phi$.

The relation between the Perron--Frobenius theory and the pseudo-Anosov theory is provided by the following; see 
e.g. \cite[Expos\'{e} 10]{FLP} or \cite[pp. 101--102]{CB}.

\begin{proposition} If $\phi:S \to S$ is pseudo-Anosov and $\mathcal R$ is a Markov partition for $\phi$ with
transition matrix $A$, then $A$ is a Perron--Frobenius matrix and $\lambda(A) = \lambda(\phi)$.
\end{proposition}

Let $\phi:S \to S$ be a pseudo-Anosov homeomorphism and let $\mathcal R$ be a Markov partition for $\phi$.  We say that a
rectangle $R$ of $\mathcal R$ is \emph{unmixed by $\phi$} if $\phi$ maps $R$ homeomorphically onto a rectangle $R' \in
\mathcal R$, and we say that it is \emph{mixed by $\phi$} otherwise.

Given $R \in \mathcal R$, any rectangle $R' \in \mathcal R$ for which $\phi^{-1}(\intr R') \cap R \neq \emptyset$ is
called a \textit{target rectangle} of $R$ for $\phi$.  The \textit{degree of $R \in \mathcal R$ for $\phi$}
is defined as
\[ \deg(R,\phi) = \left| \phi^{-1}\left( \bigcup_{R' \in \mathcal R} \intr(R') \right) \bigcap R \right|\]
Informally, $\deg(R,\phi)$ records the total number of rectangles over which $R$ maps (counted with multiplicity). The
\textit{codegree} of a rectangle $R \in \mathcal R$ for $\phi$ is the sum
\[ \codeg(R,\phi) = \sum_{R'} \deg(R',\phi) \]
where the sum is taken over all rectangles $R' \in \mathcal R$ having $R$ as a target rectangle.

We will omit the dependence on $\phi$ when it is clear from context, and will simply write $\deg(R)$ and $\codeg(R)$.
However, it will sometimes be important to make this distinction since a Markov partition for $\phi$ is also a Markov
partition for every nontrivial power of $\phi$.

\medskip
It is immediate from the definitions that if $\Gamma$ is the adjacency matrix associated to $(\phi,\mathcal R)$, and
$v_R$ is the vertex associated to the rectangle $R \in \mathcal R$, then $\deg(R,\phi) = \deg_{\textnormal{out}}(v_R)$ and
$\codeg(R,\phi) = \deg_{\textnormal{in}}(v_R)$. Translating other standard properties of adjacency graphs for Perron--Frobenius
matrices in terms of Markov partitions, we obtain the following.

\begin{proposition} \label{P:graph/partition}
Suppose $\phi:S \to S$ is pseudo-Anosov, $\mathcal R$ is a Markov partition for $\phi$, and $\Gamma$ is the adjacency
graph for the associated transition matrix.  Let $R \in \mathcal R$, and let $v_R$ be the associated vertex of $\Gamma$.  The number of distinct combinatorial directed paths in $\Gamma$ that have length $k$ and that
emanate from $v_R$ is precisely $\deg(R,\phi^k)$.  The number of distinct combinatorial paths of length $k$
leading to $v_R$ is precisely $\codeg(R,\phi^k)$.
\end{proposition}

Likewise, the following is immediate from the definitions and the related properties of Perron--Frobenius matrices.

\begin{proposition} \label{P:static char}
Let $\phi:S \to S$ be a pseudo-Anosov homeomorphism and let $\mathcal R$ be a Markov partition for $\phi$.  A rectangle is unmixed by $\phi$ if and
only if $\deg(R,\phi) = 1$ and $\codeg(R',\phi) = 1$ for the unique target rectangle $R'$ of $R$ for $\phi$.

For a positive integer $k > 0$, if $\deg(R,\phi^k) > 1$, then $\deg(R,\phi^j) > 1$ for all $j \geq k$.  Similarly, for
any $k > 0$, if $\codeg(R',\phi^k) > 1$ for some target rectangle $R'$ of $R$ by $\phi^k$, then the same is true for
some target rectangle of $R$ by $\phi^j$ for all $j \geq k$.
\end{proposition}

Proposition \ref{P:static char} immediately implies the following.

\begin{corollary} \label{C:persistent stretch}
Suppose $\phi:S \to S$ is pseudo-Anosov and $\mathcal R$ is a Markov partition.  Then if $R \in \mathcal R$ is mixed by
$\phi^k$, then it is mixed by $\phi^j$ for all $j \geq k$.
\end{corollary}

We can also translate Lemma \ref{L:matrix lemma} into the context of Markov partitions for small dilatation pseudo-Anosov homeomorphisms.

\begin{lemma}
\label{L:finite nonhomeo} There is an integer $C = C(P) > 0$, depending only on $P$, so that if $(\phi:S \to S) \in
\Psi_P$ and if $\mathcal R = \{ R_i \}$ is a small Markov partition for $\phi$, then the number of rectangles of
$\mathcal R$ that are mixed by $\phi$ is at most $C$.  Moreover, the sum of the degrees of the mixed rectangles is at most $C$, and the sum of the codegrees of all targets of mixed rectangles is at most $C$.
\end{lemma}

\begin{proof}

By the definition of a small Markov partition, we have $|\mathcal R| \leq 9|\chi(S)|$, and so
\[ P^{\frac{1}{|\chi(S)|}} \leq P^{\frac{9}{|\mathcal R|}}.\]
Then, by the definition of $\Psi_P$, we have $\lambda(\phi) \leq P^{\frac{1}{|\chi(S)|}} \leq P^{\frac{9}{|\mathcal
R|}}$.

Let $A$ be the transition matrix for the pair $(\phi,\mathcal R)$, and let $\Gamma$ be the corresponding adjacency
graph.  We denote by $v_R$ the vertex of $\Gamma$ associated to a rectangle $R \in \mathcal R$.  Since $\lambda(A) =
\lambda(\phi) \leq P^{\frac{9}{|\mathcal R|}}$, Lemma~\ref{L:matrix lemma} implies that
\begin{equation} \label{eq:bounding in and out}
 \sum_{v \in \Gamma^{(0)}} (\deg_{\textnormal{out}}(v) - 1) = \sum_{v \in \Gamma^{(0)}} (\deg_{\textnormal{in}}(v) - 1) \leq P^9
 \end{equation}
from which it follows that $\Gamma$ has at most $P^9$ vertices $v_R$ with $\deg(R,\phi) =
\deg_{\textnormal{out}}(v_R) > 1$ and at most $P^9$ vertices $v_R$ with  $\codeg(R,\phi) = \deg_{\textnormal{in}}(v_R)
> 1$.  That is, $\deg(R,\phi) = 1$ for all but at most $P^9$ rectangles, and, among the rectangles $R$ with
$\deg(R,\phi) = 1$, all but at most $P^9$ of these have $\codeg(R',\phi) = 1$ for their unique target rectangle $R'$.
Therefore, by Proposition \ref{P:static char} there are at most $2P^9$ mixed rectangles.

Finally, observe that for any $R \in \mathcal R$ we have
\[ \deg(R,\phi) = \deg_{\textnormal{out}}(v_R)= (\deg_{\textnormal{out}}(v_R) - 1) + 1 \leq P^9 + 1 \]
and similarly $\codeg(R,\phi) \leq P^9 + 1$.  Since there are at most $2P^9$ rectangles that are mixed, the sum of the
degrees of the mixed rectangles is at most $2 P^9 (P^9+1)$,  and similarly the sums of the codegrees of target
rectangles of mixed rectangles is at most $2 P^9 (P^9+1)$.

Setting $C = 2P^9(P^9 + 1)$ completes the proof.
\end{proof}

Given a rectangle $R$ of a Markov partition, we let $\ell(R)$ denote its \textit{length}, which is the length of either side of $\partial_h R$ with respect to the cone metric associated to $q$.  Similarly, we let $w(R)$ denote its \textit{width}, which is the length of either side of $\partial_v R$ with respect to $q$.

\begin{lemma} \label{L:bounded distortion}
Let $(\phi:S \to S) \in \Psi_P$ and let $\mathcal R$ be a small Markov partition for $\phi$. If $R$ and $R'$ are any two
rectangles of $\mathcal R$, then
\[ P^{-9}\ \leq\ \frac{\ell(R)}{\ell(R')} \ \leq\ P^9\quad \mbox{and}\quad  P^{-9}\ \leq\ \frac{w(R)}{w(R')}\ \leq\ P^9. \]
\end{lemma}

\begin{proof}

It suffices to show that $\ell(R)/\ell(R') \leq P^9$.  By replacing $\phi$ with $\phi^{-1}$ and switching the roles of
$R$ and $R'$, the other inequalities will follow.

Let $A$ be the transition matrix for $(\phi,\mathcal R)$ and let $\Gamma$ be its adjacency graph.   For some $k \leq
|\mathcal R|$ the $(i,j)$--entry of $A^k$ is positive, which means that $R_j$ is a target rectangle of $R_i$ for
$\phi^k$. That is, $\phi^k$ stretches $R_i$ over $R_j$. In particular
\[ \ell(\phi^{|\mathcal R|}(R_i)) \geq \ell(\phi^k(R_i)) \geq \ell(R_j)\]
and so it suffices to show that $\ell(\phi^{|\mathcal R|}(R_i)) \leq P^9\ell(R_i)$.  Indeed:
\[ \ell(\phi^{|\mathcal R|}(R_i)) \ = \ \lambda(\phi)^{|\mathcal R|} \ell(R_i) \ \leq \ P^{\frac{|\mathcal R|}{|\chi(S)|}}\ell(R_i) \ \leq \ P^9\ell(R_i). \]
The equality uses the definition of a Markov partition, the first inequality follows from the definition of $\Psi_P$,
and the second inequality comes from the definition of a small Markov partition.
\end{proof}

%%%%%%%%%%%%%%%%%%%%%%%%%%%%%%%%%%%%%%%%%%%%%%%%%%%%%%%%%%%%%%%%%%%%%%%%%%%%%%%%%%%%%%%%%
%%%%%%%%%%%%%%%%%%%%%%%%%%%%%%%%%%%%%%%%%%%%%%%%%%%%%%%%%%%%%%%%%%%%%%%%%%%%%%%%%%%%%%%%%
%%%%%%%%%%%%%%%%%%%%%%%%%%%%%%%%%%%%%%%%%%%%%%%%%%%%%%%%%%%%%%%%%%%%%%%%%%%%%%%%%%%%%%%%%
%%%%%%%%%%%%%%%%%%%%%%%%%%%%%%%%%%%%%%%%%%%%%%%%%%%%%%%%%%%%%%%%%%%%%%%%%%%%%%%%%%%%%%%%%
%%%%%%%%%%%%%%%%%%%%%%%%%%%%%%%%%%%%%%%%%%%%%%%%%%%%%%%%%%%%%%%%%%%%%%%%%%%%%%%%%%%%%%%%%
%%%%%%%%%%%%%%%%%%%%%%%%%%%%%%%%%%%%%%%%%%%%%%%%%%%

%%%%%%%%%%%%%%%%%%%%%%%%%%%%%%%%%
%%%%%%%%%%%%%%%%%%%%%%%%%%%%%%%%%
\section{Cell structures on $S$} \label{S:cell structures}
%%%%%%%%%%%%%%%%%%%%%%%%%%%%%%%%%
%%%%%%%%%%%%%%%%%%%%%%%%%%%%%%%%%

Let $\phi$ be a pseudo-Anosov homeomorphism of $S$, and let $\mathcal R = \{ R_i \}$ be a Markov partition for $\phi$.
The Markov partition determines a cell structure on $S$, which we denote by $X = X(\mathcal R)$, as follows.  The
vertices of $X$ are the corners of the rectangles $R_i$ together with the marked points and the singular points of the
stable (and unstable) foliation.  The vertices divide the boundary of each rectangle of $\mathcal R$ into a union of arcs, which we
declare to be the $1$--cells of our complex.  Finally, the 2--cells are simply the rectangles themselves.  We refer to
$1$--cells as either \textit{horizontal} or \textit{vertical} according to which foliation they lie in.

\medskip

For the next lemma, recall the notion of $K$--bounded complexity, which is defined in Section~\ref{S:finiteness}.

\begin{lemma}
\label{L:bounded complexity 1} There is an integer $D_1=D_1(P) > 0$ so that for any $\phi \in \Psi_P$, and for any
small Markov partition $\mathcal R$ for $\phi$, each of the cells of $X(\mathcal R)$ has $D_1$--bounded
complexity
\end{lemma}

\begin{proof}
Let $R \in \mathcal R$.  The first observation is that each of the four sides of $R$ contains at most one marked point
or one singularity of the stable foliation for $\phi$.  Indeed, if a side of $\partial R$ were to contain more than one
singularity or marked point, then the vertical or horizontal interval of $\partial R$ connecting these two points would
be a leaf of either the stable or unstable foliation.  This is a contradiction since some power of $\phi$ would have to
take this segment to itself, but $\phi$ nontrivially stretches or shrinks all leaves of the given foliation.

Now, for each $1$--cell $e$ in the boundary of $R$, at least one of the following holds:
\begin{itemize}
\item[\emph{(i)}] $e$ is an entire edge of the boundary of a rectangle, or
\item[\emph{(ii)}]  $e$ has a corner of $R$ as an endpoint, or
\item[\emph{(iii)}] $e$ has a marked point or singularity of the stable foliation for $\phi$ as an endpoint.
\end{itemize}
Along each side of $R$ there are at most two 1--cells of type \emph{(ii)} and, since there is at most one singularity
or marked point, there are at most two 1--cells of type \emph{(iii)}. By Lemma~\ref{L:bounded distortion}, the length
of any rectangle of $\mathcal R$ is at least $P^{-9}\ell(R)$ and the width is at least $P^{-9}w(R)$.  Thus, there are at most $P^9$ $1$--cells of type \emph{(i)} along each side of $R$. Summarizing,
there are at most $(P^9+4)$ $1$--cells along each side of $R$, and hence at most $(4P^9+16)$ $1$--cells in $\partial
R$. Thus, the constant $D_1=4P^9+16$ satisfies the conclusion of the lemma (this is trivial for $0$-- and $1$--cells).
\end{proof}

For any rectangle $R \in \mathcal R$, its image $\phi(R)$ is a union of subrectangles of rectangles in $\mathcal R$.  Thus, if we
define
\[ \phi(\mathcal R) = \{ \phi(R) \, | \, R \in \mathcal R \}\]
then $\mathcal R \cup \phi(\mathcal R)$ determines a rectangle decomposition of
$S$, with all rectangles being subrectangles of rectangles in $\mathcal R$. We define another cell structure $Y =
Y(\phi,\mathcal R)$ on $S$, in exactly the same way as above, using this new rectangle decomposition; see Figure
\ref{F:markovsubdivide}.

\begin{figure}[htb]
\centerline{}\centerline{}
\begin{center}
\ \psfig{file=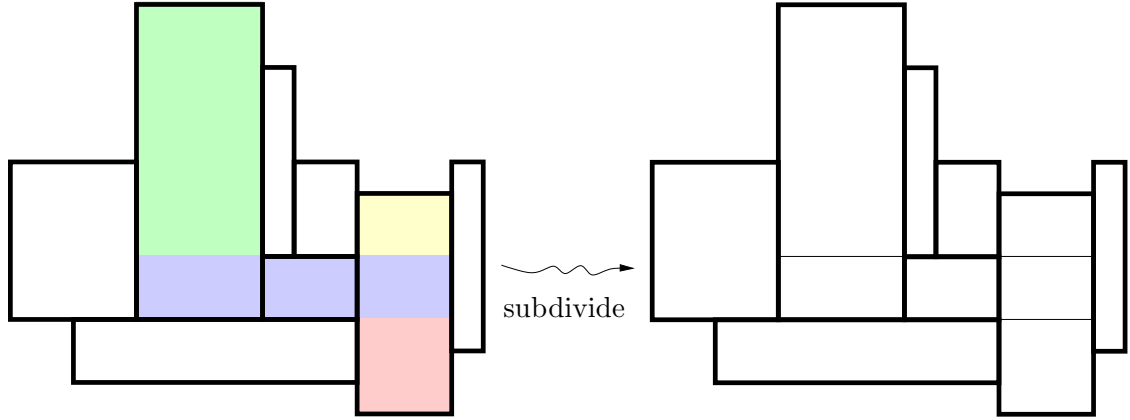,height=2.2truein} \caption{The image of the mixed rectangles are colored (see Figure \ref{F:markovpic1}) and they determine the subdivision of the original rectangles into subrectangles.}  \label{F:markovsubdivide}
\end{center}
 \setlength{\unitlength}{1in}
  \begin{picture}(0,0)(0,0)
   \put(2.7,1.4){subdivide}
  \end{picture}
\end{figure}

We can briefly explain the reason for introducing the cell structure $Y$.  In Section~\ref{S:3-mfd}, we will construct a cell structure on the mapping torus for $\phi$.  Each 3--cell will come from crossing a rectangle of $\mathcal R$ with the unit interval.  The gluing map for the one face of this ``box'' is given by $\phi$, and hence would not be cellular in the $X$--structure.

The following gives bounds on the complexity of the cells of $Y$ and the ``relative complexity'' of $X$ with respect to $Y$.

\begin{lemma} \label{L:bounded complexity 2}
There is an integer $D_2=D_2(P)$ so that if $\phi \in \Psi_P$ and $\mathcal R=\{R_i\}$ is a small Markov partition
for $\phi$, then the following hold.
\begin{enumerate}
\item \label{part1} Each $1$--cell of $X= X(\mathcal R)$ is a union of at most $D_2$ $1$--cells of $Y$.
\item \label{part2} Each $2$--cell of $X$ is a union of at most $D_2$ $2$--cells of $Y$.
\item \label{part3} For each $2$--cell $e$ of $X$, $\phi(e)$ is a union of at most $D_2$ $2$--cells of $Y$.
\item \label{part4} For each $1$--cell $e$ of $X$, $\phi(e)$ is a union of at most $D_2$ $1$--cells of $Y$.
\item \label{part5} Each of the cells of $Y(\phi,\mathcal R)$ has $D_2$--bounded complexity.\end{enumerate}
\end{lemma}

\begin{proof}

By the definition of a Markov partition, the cell structure $Y=Y(\phi,\mathcal R)$ is obtained from the cell structure
$X=X(\mathcal R)$ by subdividing each rectangle of $X$ into subrectangles along horizontal arcs.  Given $R \in
\mathcal R$, the number of rectangles in the subdivision is precisely $\codeg(R)$, and is thus bounded by $C = C(P)$
according to Lemma \ref{L:finite nonhomeo}. Therefore, part \ref{part2} holds for $D_2 \geq C$.
Similarly, given any rectangle $R$, $\phi(R)$ is a union of $\deg(R)\leq C$ $2$--cells of
$Y$, and so part \ref{part3} holds for $D_2 \geq C$.

The subdivision of $R$ is obtained by first subdividing the vertical sides of $R$ by adding at most $2(\codeg(R) -1)$
new $0$--cells, then adding $\codeg(R) - 1$ horizontal $1$--cells from the left side of $R$ to the right. Therefore,
since each vertical $1$--cell of $X$ lies in exactly two rectangles $R$ and $R'$, it is subdivided into at most
$(\codeg(R) + \codeg(R')) \leq 2C$ $1$--cells of $Y$ by Lemma \ref{L:finite nonhomeo}. Every horizontal $1$--cell of
$X$ is still a horizontal $1$--cell of $Y$, so part \ref{part1} holds for $D_2 \geq 2C$.

For part \ref{part5}, observe that any $2$--cell $e$ of $Y$ is a subrectangle of some rectangle $R \in \mathcal R$.
Since each rectangle $R$ is a $2$--cell of $X$, it has at most $D_1$ $1$--cells of $X$ in its boundary.  By the
previous paragraph, each of these $1$--cells is subdivided into at most $2C$ $1$--cells of $Y$.  Therefore, there are
at most $2CD_1$ $1$--cells of $Y$ in the boundary of $R$, and thus no more than this number in the boundary of $e$.
Part \ref{part5} holds for any $D_2 \geq 2CD_1$ (this is trivial for $0$-- and $1$--cells).

For part \ref{part4}, observe that if $e$ is a vertical $1$--cell of $X$, then $\phi(e)$ is contained in the vertical boundary
of some rectangle $R \in \mathcal R$.  As mentioned in the previous paragraph, the boundary of any rectangle has at most $2CD_1$ $1$--cells of $Y$ in its boundary, so $\phi(e)$ is a union of at most $2CD_1$ $1$--cells of $Y$.
If $e$ is a horizontal $1$--cell, then $e$ is contained in the horizontal boundary of some rectangle $R \in
\mathcal R$, and so $\phi(e)$ is contained in the union of at most $\deg(R) \leq C$ horizontal boundaries of rectangles
of the subdivision.  As already mentioned, the horizontal boundary of each of these rectangles is either one of the
added $1$--cells, or else a horizontal boundary of some rectangle of $\mathcal R$, which contains at most $D_1$
$1$--cells of $X$ (which are also the $1$--cells of $Y$ as they are contained in the horizontal edges of a rectangle).
Therefore, $\phi(e)$ contains at most $C D_1$ $1$--cells of $Y$.  Therefore, part \ref{part4} holds for any $D_2 \geq
2CD_1$.

The lemma now follows by setting $D_2 = 2CD_1$.
\end{proof}

%%%%%%%%%%%%%%%%%%%%%%%%%%%%%%%%%%%%%%%%%%%%%%%%%%%%%%%%%%%%%%%%%%%%%%%%%%%%%%%%%%%%%%%%%
%%%%%%%%%%%%%%%%%%%%%%%%%%%%%%%%%%%%%%%%%%%%%%%%%%%%%%%%%%%%%%%%%%%%%%%%%%%%%%%%%%%%%%%%%
%%%%%%%%%%%%%%%%%%%%%%%%%%%%%%%%%%%%%%%%%%%%%%%%%%%%%%%%%%%%%%%%%%%%%%%%%%%%%%%%%%%%%%%%%
%%%%%%%%%%%%%%%%%%%%%%%%%%%%%%%%%%%%%%%%%%%%%%%%%%%%%%%%%%%%%%%%%%%%%%%%%%%%%%%%%%%%%%%%%
%%%%%%%%%%%%%%%%%%%%%%%%%%%%%%%%%%%%%%%%%%%%%%%%%%%%%%%%%%%%%%%%%%%%%%%%%%%%%%%%%%%%%%%%%

%%%%%%%%%%%%%%%%%%%%%%%%%%%%%%%%%%%%%%%%%%%%%%%%%%%
%%%%%%%%%%%%%%%%%%%%%%%%%%%%%%%%%%%%%%%%%%%%%%%%%%%
\section{Equivalence relations on rectangles} \label{S:equiv rel}
%%%%%%%%%%%%%%%%%%%%%%%%%%%%%%%%%%%%%%%%%%%%%%%%%%%
%%%%%%%%%%%%%%%%%%%%%%%%%%%%%%%%%%%%%%%%%%%%%%%%%%%

Given a pseudo-Anosov $(\phi:S \to S) \in \Psi_P$, let $\mathcal R$ be a small Markov partition for $\phi$.  In this
section we will use this data to construct a quotient space $\pi:S \to T$, and a cell structure $W$ on $T$ for which $\pi$ is
cellular with respect to $Y$. Moreover, we will prove that $W$ has uniformly bounded complexity.

The quotient $T$ will be obtained by gluing together rectangles of $\mathcal R$ via a certain equivalence relation.
This relation has the property that if a rectangle $R$ is equivalent to $R'$, then there is a (unique) power
$\phi^{\beta(R,R')}$ that takes $R$ homeomorphically onto the rectangle $R'$.  Moreover, $\phi^{\beta(R,R')}$ will be
cellular with respect to $Y$--cell structures on $R$ and $R'$, respectively.

This equivalence relation is most easily constructed from simpler equivalence relations.  Our approach will be to
define a first approximation to the equivalence relation we are searching for, and then refine it (twice) to achieve
the equivalence relation with the required properties.  Along the way we verify other properties that will be needed later.

%%%%%%%%%%%%%%%%%%%%%%%%%%%%%%%%%%%%%%%%%%%%%%%%%%%%%%
%%%%%%%%%%%%%%%%%%%%%%%%%%%%%%%%%%%%%%%%%%%%%%%%%%%%%%
\subsection{The first approximation: $h$--equivalence}
%%%%%%%%%%%%%%%%%%%%%%%%%%%%%%%%%%%%%%%%%%%%%%%%%%%%%%
%%%%%%%%%%%%%%%%%%%%%%%%%%%%%%%%%%%%%%%%%%%%%%%%%%%%%%

Define an equivalence relation $\stackrel{h}{\sim}$ on $\mathcal R$ by declaring
\[ R \stackrel{h}{\sim} R' \]
if there exists $\beta \in \mathbb Z$ so that $\phi^\beta$ takes $R$ homeomorphically onto $R'$. The `h' stands for
``homeomorphically''. As no rectangle can be taken homeomorphically to itself by a nontrivial power of $\phi$, it
follows that if $R \stackrel{h}{\sim} R'$ then there exists a unique integer $\beta(R,R')$ for which
$\phi^{\beta(R,R')}(R) = R'$.

\begin{proposition} \label{P:relation 1 prop}
Let $\phi \in \Psi_P$, and let $\mathcal R$ be a small Markov partition for $\phi$.  The $h$--equivalence classes in $\mathcal R$ have the form
\[ \{\phi^j(R)\}_{j = 0}^k\]
where each $k = k(R)$ is an integer.  There is a constant $E_h = E_h(P)$ so that there are at most $E_h$
$h$--equivalence classes.
\end{proposition}

\begin{proof}

We first observe that $R \stackrel{h}{\sim} R'$ if and only if $R$ is unmixed for $\phi^j$ and $\phi^j(R) = R'$.   Proposition \ref{P:static char} then
implies that $R \stackrel{h}{\sim} R'$ if and only if $\deg(R,\phi^j)=1$,  $R'$ is the
unique target rectangle of $R$ and $\codeg(R',\phi^j)=1$. Corollary \ref{C:persistent stretch} now implies that the $h$--equivalence classes have the required form.

To prove the second statement of the proposition, we observe that, according to the description of the $h$--equivalence relation provided in the previous paragraph, the last rectangle $R$ in any given equivalence class (which is well-defined by the first part of the proposition) satisfies one of the following two conditions:
\begin{itemize}
\item[\emph{(i)}] $\deg(R,\phi) > 1$.
\item[\emph{(ii)}] $\deg(R,\phi) = 1$ and $\codeg(R',\phi) > 1$, where $R'$ is the unique target rectangle of $R$.
\end{itemize}
The number of rectangles $R$ of type \emph{(i)} is bounded above by the constant $C=C(P)$ from Lemma~\ref{L:finite nonhomeo}.
The number of type \emph{(ii)} is bounded above by the sum of the codegrees of rectangles with codegree greater than 1.  Again, by Lemma~\ref{L:finite nonhomeo}, this is bounded by $C$.  Thus, we may take $E_h$ to be $2C$.
\end{proof}

Given $(\phi:S \to S) \in \Psi_P$ and $\mathcal R$ a small Markov partition for $\phi$, we index the rectangles of $\mathcal R$
as $\{R_i\}_{i = 1}^n$ so that
\begin{enumerate}
\item The $h$--equivalence classes all have the form $\{R_i,R_{i+1},\dots,R_{i+k}\}$.
\item If $\{R_i,R_{i+1},\ldots,R_{i+k}\}$ is an equivalence class, then $R_{i+j} = \phi^j(R_i)$.
\end{enumerate}
That this is possible follows from Proposition \ref{P:relation 1 prop}.

%%%%%%%%%%%%%%%%%%%%%%%%%%%%%%%%%%%%%%%%%%%%%%%%%%%%%%
%%%%%%%%%%%%%%%%%%%%%%%%%%%%%%%%%%%%%%%%%%%%%%%%%%%%%%
\subsection{The second approximation: $N$--equivalence}
%%%%%%%%%%%%%%%%%%%%%%%%%%%%%%%%%%%%%%%%%%%%%%%%%%%%%%
%%%%%%%%%%%%%%%%%%%%%%%%%%%%%%%%%%%%%%%%%%%%%%%%%%%%%%

Say that two rectangles $R, R' \in \mathcal R$ are \textit{adjacent} if $R \cap R'$ contains at least one point that is not a singular
point or a marked point. Observe that, by definition, a rectangle is adjacent to itself.  We let $N_1(R) \subset
\mathcal R$ denote the set of rectangles adjacent to $R$.  We think of $N_1(R)$ as the ``1--neighborhood'' of $R$.

\begin{lemma} \label{L:rect in N1}
There exists a constant $D_3 = D_3(P)$ so that if $(\phi:S \to S) \in \Psi_P$ and $\mathcal R$ is a small Markov
partition for $\phi$, then the number of rectangles in $N_1(R)$ is at most $D_3$ for any $R \in \mathcal R$.
\end{lemma}
\begin{proof}
The number of rectangles in $N_1(R)$ is at most the number of $1$--cells of $X$ in the boundary of $R$ plus $5$.  This
is because each $1$--cell contributes at most one new adjacent rectangle to $N_1(R)$, each corner vertex of $R$ contributes at most one more rectangle, and $R$ itself contributes $1$.  From Lemma
\ref{L:bounded complexity 2} it follows that we can take $D_3 = D_2 + 5$.
\end{proof}

Now we define a refinement of $h$--equivalence, called $N$--equivalence,  by declaring
\[ R \stackrel{N}{\sim} R'\]
if $\phi^{\beta(R,R')}$ does not mix $R''$ for any $R'' \in N_1(R)$ (the `N' stands for 1--neighborhood).  That is, $\phi^{\beta(R,R')}$
maps the entire 1--neighborhood $N_1(R)$ to the 1--neighborhood of $R'$, taking each rectangle homeomorphically onto another rectangle.
We leave it to the reader to check that $N$--equivalence is a refinement of $h$--equivalence.

Figure \ref{F:equivalence} shows the local picture of the surface after applying $\phi^j$, for $j =0,1,2,3$.  We have
labeled a few of the rectangles:
\[ R_i,\ \ R_{i+1} = \phi(R_i), \ \ R_{i+2} = \phi^2(R_i), \ \ R_{i+3} = \phi^3(R_i) \ \ \mbox{ and }\]
\[ R_j,\ \ R_{j+1} = \phi(R_j),\ \  R_{j+2} = \phi^2(R_j),\ \  R_{j+3} \supsetneq \phi^3(R_j).\]
The rectangles are related as follows:
\[ R_i \stackrel{h}{\sim} R_{i+1} \stackrel{h}{\sim} R_{i+2} \stackrel{h}{\sim} R_{i+3} \]
\[ R_i \stackrel{N}{\not \sim} R_{i+1} \stackrel{N}{\sim} R_{i+2} \stackrel{N}{\not \sim} R_{i+3} \]
(also $R_i \stackrel{N}{\not \sim} R_{i+3}$, ), and
\[ R_j \stackrel{h}{\sim} R_{j+1} \stackrel{h}{\sim} R_{j+2} \stackrel{h}{\not \sim} R_{j+3} \]
\[ R_j \stackrel{N}{\sim} R_{j+1} \stackrel{N}{\sim} R_{j+2} \stackrel{N}{\not \sim} R_{j+3}. \]
Of course, each rectangle is equivalent to itself with respect to either relation.

\begin{figure}[htb]
\centerline{}\centerline{}
\begin{center}
\ \psfig{file=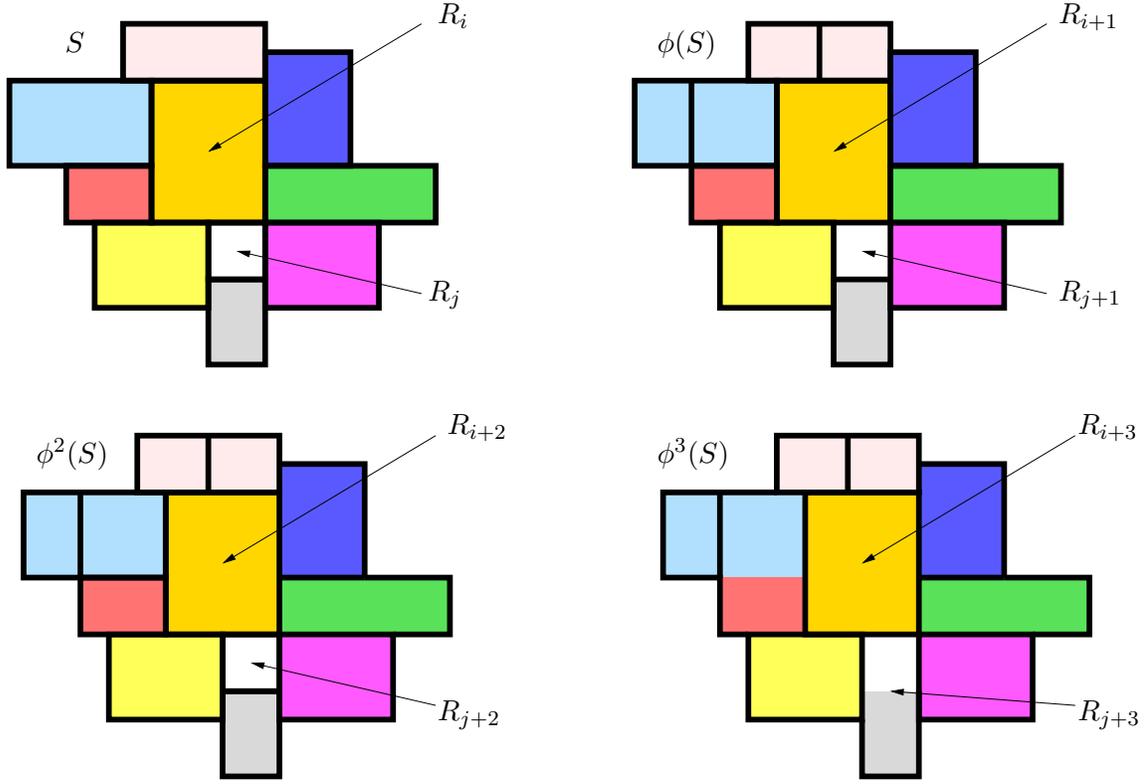,height=4truein} \caption{A local picture of the surface after applying $\phi^j$, $j =
0,1,2,3$.}  \label{F:equivalence}
\end{center}
  \setlength{\unitlength}{1in}
  \begin{picture}(0,0)(0,0)
    \put(2.45,4.65){$R_i$}
    \put(5.7,4.65){$R_{i+1}$}
    \put(2.5,2.5){$R_{i+2}$}
    \put(5.8,2.5){$R_{i+3}$}
    \put(.5,4.5){$S$}
    \put(3.6,4.5){$\phi(S)$}
    \put(.35,2.35){$\phi^2(S)$}
    \put(3.6,2.35){$\phi^3(S)$}
    \put(2.4,3.2){$R_j$}
    \put(5.7,3.2){$R_{j+1}$}
    \put(2.45,1){$R_{j+2}$}
    \put(5.8,1){$R_{j+3}$}
  \end{picture}
\end{figure}

\begin{proposition} \label{P:relation 2 prop}
Let $\phi \in \Psi_P$, and let $\mathcal R$ be a small Markov partition for $\phi$.  Each $N$--equivalence
class has the form
\[\{R_i,R_{i+1},\dots,R_{i+s} \} = \{R_i,\phi(R_i),\dots,\phi^s(R_i)\}\]
for some $s = s(R_i) \in \mathbb Z$.  There is a constant $E_N = E_N(P)$ so that $\mathcal R$ has at most $E_N$
$N$--equivalence classes.
\end{proposition}

\begin{proof}

Similarly to the proof of Proposition~\ref{P:relation 1 prop}, the first statement follows from Corollary \ref{C:persistent stretch}.
Now, consider an arbitrary $h$--equivalence class:
\[ \{R_i,R_{i+1},\dots,R_{i+k}\} = \{R_i,\phi(R_i),\dots,\phi^k(R_i)\}. \]
This is partitioned into its $N$--equivalence classes as follows.  The class divides at $R_{i+j}$ (that is, $R_{i+j+1}$ begins a new $N$--equivalence class) if and only if
$N_1(R_{i+j})$ contains a rectangle that is mixed by $\phi$.
By Lemma~\ref{L:finite nonhomeo}, at most $C=C(P)$ rectangles of $\mathcal R$ are mixed.
Moreover, according to Lemma \ref{L:rect in N1}, each rectangle---in particular, each mixed rectangle---is a
$1$--neighbor to at most $D_3=D_3(P)$ rectangles.  Therefore, there are at most $CD_3$ rectangles that are $1$--neighbors of mixed rectangles.
Thus, each $h$--equivalence class can be subdivided into at
most $(CD_3 + 1)$ $N$--equivalence classes.  According to Proposition \ref{P:relation 1 prop} the number of $h$--equivalence classes is at most $E_h$, and thus the proposition follows if we take $E_N = E_h(CD_3+1)$.
\end{proof}

The next proposition explains one advantage of the $N$--equivalence relation over the $h$--equivalence relation.

\begin{proposition} \label{P:sim2 and X}
Let $(\phi:S \to S) \in \Psi_P$ and let $\mathcal R$ be a small Markov partition for $\phi$.  If $R \stackrel{N}{\sim} R'$,
then $\phi^{\beta(R,R')}|_R:R \to R'$ is cellular with respect to $X$.
\end{proposition}
\begin{proof}
The cell structure $X$ is defined using the rectangles of $\mathcal R$, together with the way adjacent rectangles intersect one another, and the singular and marked points.  Since $\phi$ preserves the singular and marked points, and since $R \stackrel{N}{\sim} R'$ implies
$\phi^{\beta(R,R')}$ maps all rectangles of $N_1(R)$ homeomorphically onto a rectangle in $N_1(R')$, the result
follows.
\end{proof}

%%%%%%%%%%%%%%%%%%%%%%%%%%%%%%%%%%%%%%%%%%%%%%%%
%%%%%%%%%%%%%%%%%%%%%%%%%%%%%%%%%%%%%%%%%%%%%%%%
\subsection{The ``right'' relation on rectangles: $Y$--equivalence}
%%%%%%%%%%%%%%%%%%%%%%%%%%%%%%%%%%%%%%%%%%%%%%%%
%%%%%%%%%%%%%%%%%%%%%%%%%%%%%%%%%%%%%%%%%%%%%%%%

We define a refinement of $N$--equivalence, called $Y$--equivalence, by dividing each $N$--equivalence class
$\{R_i,R_{i+1},\dots,R_{i+k}\}$ with more than one element into two $Y$--equivalence classes $\{R_i\}$ and
$\{R_{i+1},R_{i+2},\dots,R_{i+k}\}$. That is, we split off the initial element of each $N$--equivalence class
into its own $Y$--equivalence class.  The $Y$--equivalence classes are also consecutive with respect to the indices
and so we can refer to the initial and terminal rectangles of a $Y$--equivalence class.  We write $R \stackrel{Y}{\sim} R'$ if $R$ and $R'$ are $Y$--equivalent.

The new feature of $Y$--equivalence is that any time two distinct rectangles are $Y$--equivalent, they both have
codegree one.  In fact, all rectangles in the $1$--neighborhood are $\phi$--images of unmixed rectangles.  It follows
that this equivalence relation behaves nicely with respect to the cell structure $Y$, hence the terminology; see
Proposition~\ref{P:sim1 and Y} below.

If we take $E_Y=2 E_N$, we immediately obtain the following consequence of Proposition~\ref{P:relation 2 prop}.

\begin{corollary} \label{C:relation 1 prop}
There is a constant $E_Y=E_Y(P)$ so that for any $\phi \in \Psi_P$, any small Markov partition $\mathcal R$ for
$\phi$ has at most $E_Y$ $Y$--equivalence classes.
\end{corollary}

The next proposition is the analogue  of Proposition \ref{P:sim2 and X} for $Y$--equivalence and the cell structure
$Y$.

\begin{proposition} \label{P:sim1 and Y}
Let $(\phi:S \to S) \in \Psi_P$ and let $\mathcal R$ be a small Markov partition for $\phi$.  If the $Y$--equivalence class of $R$ contains more than one element, and $R
\stackrel{Y}{\sim} R'$, then the $X$-- and $Y$--cell structures on $R$ agree, and
likewise for $R'$.  Moreover, $\phi^{\beta(R,R')}|_R:R \to R'$ is cellular with respect to $Y$.
\end{proposition}

The assumption that the $Y$--equivalence class of $R$ contains more than one element is necessary for the first statement, since otherwise it would follow that the $X$ and
$Y$ cell structures coincide on all of $S$.  This would imply that $\phi$ is cellular with respect to $X$,  and hence
finite order, which is absurd.

\begin{proof}

Since the $Y$--equivalence class of $R$ and $R'$ contains more than one element, neither $R$ nor $R'$ is the initial rectangle of the $N$--equivalence class they lie in.  Thus, $\phi^{-1}(R)$ and $\phi^{-1}(R')$ are both elements of $\mathcal R$ and
\[ \phi^{-1}(R) \ \stackrel{N}{\sim} \ R \stackrel{N}{\sim} \ \phi^{-1}(R') \ \stackrel{N}{\sim} \ R'.  \]
By Proposition \ref{P:sim2 and X}, the maps
\[ \phi|_{\phi^{-1}(R)}:\phi^{-1}(R) \to R \ \ \textnormal{and}\ \ \phi|_{\phi^{-1}(R')}:\phi^{-1}(R') \to R' \]
are both cellular with respect to $X$.  Thus, by the definition of $Y$, it follows that the $Y$--structures on $R$ and
$R'$ are exactly the same as the $X$--structures on $R$ and $R'$, respectively. As $R \stackrel{N}{\sim} R'$,
Proposition \ref{P:sim2 and X} guarantees that
\[ \phi^{\beta(R,R')}|_R:R \to R'\]
is cellular with respect to $X$, and hence also with respect to $Y$, as required.
\end{proof}

%%%%%%%%%%%%%%%%%%%%%%%%%%%%%%%%%%%%%%%%%%%
%%%%%%%%%%%%%%%%%%%%%%%%%%%%%%%%%%%%%%%%%%%
\subsection{From rectangles to the surface}
%%%%%%%%%%%%%%%%%%%%%%%%%%%%%%%%%%%%%%%%%%%
%%%%%%%%%%%%%%%%%%%%%%%%%%%%%%%%%%%%%%%%%%%

We will use the $Y$--equivalence relation to construct a quotient of $S$ by gluing $R$ to $R'$ by $\phi^{\beta(R,R')}$
whenever $R \stackrel{Y}{\sim} R'$.  To better understand this quotient, we will study the equivalence relation on $S$
that this determines.  This is most easily achieved by breaking the equivalence relation up into simpler relations as
follows.

We write $(x,R) \leftrightarrow (x',R')$ to mean that the following conditions hold:
\begin{enumerate}
\item $x \in R$
\item $x' \in R'$
\item $R \stackrel{Y}{\sim} R'$
\item $\phi^{\beta(R,R')}(x) = x'$
\end{enumerate}
Now write $x \leftrightarrow x'$ and say that $x$ {\em is related to} $x'$ if $(x,R) \leftrightarrow (x',R')$ for some $R,R' \in \mathcal R$.

The relation $\leftrightarrow$ is easily seen to be symmetric and reflexive, but it may not be transitive.
This is because if $x'$ lies in two distinct rectangles $R'$ and $R''$, it may be that $(x,R) \leftrightarrow (x',R')$ and $(x',R'') \leftrightarrow (x'',R''')$.  We let $\sim$ denote the equivalence relation on $S$ obtained from
the transitive closure of the relation $\leftrightarrow$ on $S$.

\begin{lemma} \label{L:hopping sim}
The $\sim$--equivalence class of $x$, which is contained in $\{ \phi^k(x)\}_{k \in \mathbb Z}$, consists of consecutive
$\phi$-iterates of $x$.  What is more, if $x \sim x'$, then after possibly interchanging the roles of $x$ and $x'$,
there exists $k \in \mathbb Z_{\geq 0}$, and a sequence of rectangles $R_{i_0}, \dots, R_{i_{k-1}} \in \mathcal R$, so that
\[ x \leftrightarrow \phi(x) \leftrightarrow \dots \leftrightarrow \phi^{k-1}(x) \leftrightarrow \phi^k(x) = x' \]
with $(\phi^j(x),R_{i_j}) \leftrightarrow (\phi^{j+1}(x),\phi(R_{i_j}))$ for $j = 0,\dots,k-1$.
\end{lemma}

The second part of the lemma says that when $x \sim x'$, after possibly interchanging $x$ and $x'$, we can get from $x$
to $x'$ moving forward through consecutive elements of the $\phi$--orbit by applying the relation $\leftrightarrow$. We
caution the reader that it may be necessary to interchange the roles of $x$ and $x'$, even when they lie in a periodic
orbit.

\begin{remark}
For a periodic point, the orbit $\{ \phi^k(x)\}_{k \in \mathbb Z}$ is a finite set, in which case the ordering is a
cyclic ordering.  However, it still makes sense to say that a set consists of consecutive $\phi$-iterates of $x$.
\end{remark}

\begin{proof}[Proof of Lemma~\ref{L:hopping sim}]
Note that if $(x,R) \leftrightarrow (x',R')$ with $\beta(R,R') \geq 0$ then we also have
\[(x,R) \leftrightarrow (\phi^j(x),\phi^j(R)) \]
for $j = 0,\dots,\beta(R,R')$, since $\{R,\phi(R),\dots,\phi^{\beta(R,R')}(R)\}$ is contained in a single
$Y$--equivalence class (that is, the $Y$--equivalence classes of rectangles are consecutive). Since $\sim$ is the
transitive closure of $\leftrightarrow$, the full equivalence class is obtained by stringing together these sets
whenever they intersect, and it follows that the equivalence class of $x$ consists of consecutive $\phi$-iterates of
$x$, and further that they are related as in statement of the lemma.
\end{proof}

%%%%%%%%%%%%%%%%%%%%%%%%%%%%%%%%%%%%%%%%%%%%%%%%%%%%%%%%%%%%
%%%%%%%%%%%%%%%%%%%%%%%%%%%%%%%%%%%%%%%%%%%%%%%%%%%%%%%%%%%%
\subsection{The cell structure $Y$ and $\sim$--equivalence}
%%%%%%%%%%%%%%%%%%%%%%%%%%%%%%%%%%%%%%%%%%%%%%%%%%%%%%%%%%%%
%%%%%%%%%%%%%%%%%%%%%%%%%%%%%%%%%%%%%%%%%%%%%%%%%%%%%%%%%%%%

The main purpose of this section is to describe the structure of the $\sim$--equivalence classes and how they relate
to the $Y$ cell structure.

Given a point $x \in S$, set
\[ k_-(x) = \inf \{k \in \Z \, | \, \phi^j(x) \sim x \mbox{ for all } k \leq j \leq 0 \}\]
and
\[ k_+(x) = \sup \{ k \in \Z \, | \, \phi^j(x) \sim x \mbox{ for all } 0 \leq j \leq k \}.\]
That is, we look at all consecutive points in the orbit of $x$ that are equivalent, and take the infimum and supremum,
respectively, of the consecutive exponents (beginning at $0$) that occur.  According to Lemma \ref{L:hopping sim}, the set $\{\phi^j(x)\}_{j =
k_-(x)}^{k_+(x)}$ is the $\sim$--equivalence class of $x$. Observe that if $x$ is a fixed point of $\phi$  then $k_-(x)
= - \infty$ and $k_+(x) = + \infty$.

\begin{lemma}\label{L:orbit sim}
If $x$ is not a marked point or singular point, then $k_-(x),k_+(x) \in \mathbb Z$.  In particular, it cannot be the
case that the entire $\phi$--orbit of $x$ lies in the same $\sim$--equivalence class.
\end{lemma}

\begin{proof}

First suppose that $x$ is a periodic point of order $n$.  The only way that $k_+(x) = \infty$ or $k_-(x) = -\infty$ is
if any two points of the orbit are $\sim$--equivalent.  Up to changing $x$ within its orbit, we may assume that there
is a rectangle $R$ containing $x$ that is mixed by $\phi$; indeed, otherwise the pseudo-Anosov map $\phi^{kn}$ would
preserve the collection of rectangles containing $x$ for all $k$ which is impossible.  We now prove $x \not \sim
\phi(x)$, which will complete the proof in the case of a periodic point.

The proof is by contradiction, so assume that $x \sim \phi(x)$.  According to Lemma~\ref{L:hopping sim}, there are two cases to consider depending on whether or not the roles of $x$ and $\phi(x)$ must be interchanged.
Thus, we either have
a rectangle $R'$ for which
\[ (x,R') \leftrightarrow (\phi(x),\phi(R')) \]
or else there exists a sequence of rectangles $R_{i_0}, \dots, R_{i_{n-2}}$ so that
\[ (\phi^j(\phi(x)),R_{i_j}) \leftrightarrow (\phi^{j+1}(\phi(x)),\phi(R_{i_j}))\]
for $j = 0,\dots,n-2$.  In particular, in the second case we have $(\phi(x),R_{i_0}) \leftrightarrow (\phi^2(x),\phi(R_{i_0}))$.

The first case is clearly impossible since $R'$ is adjacent to $R$ which is mixed by $\phi$, and hence $R'
\stackrel{Y}{\not \sim} \phi(R')$.  In particular, it follows that $(x,R') \not\leftrightarrow (\phi(x),\phi(R'))$,
which is a contradiction.  In the second case, we have $R_{i_0} \stackrel{Y}{\sim} \phi(R_{i_0})$ which implies
$\phi^{-1}(R_{i_0}) \stackrel{N}{\sim} R_{i_0}$. Since $\phi^{-1}(R_{i_0})$ is a rectangle containing $x$, it is
adjacent to (the mixed rectangle) $R$ and therefore $\phi^{-1}(R_{i_0}) \stackrel{N}{\not \sim} R_{i_0}$, another contradiction.

Therefore, it must be the case that $x \not \sim \phi(x)$, as required.

\medskip

We now consider the case where $x$ is not a periodic point for $\phi$. By Lemma~\ref{L:hopping sim}, the equivalence
class of $x$ consists of consecutive $\phi$-iterates.  Let $j$ be a positive integer so that one of the rectangles
containing $\phi^j(x)$, say $R$, is mixed by $\phi$.  As in the periodic case, it follows that $\phi^j(x) \not
\leftrightarrow \phi^{j+1}(x)$.  Since $x$ is aperiodic, it follows from Lemma~\ref{L:hopping sim} that $x \not \sim
\phi^{j+1}(x)$, so $k_+(x) \leq j+1$.

Similarly, note that there is some $j > 0$ so that $\phi^{-j}(x)$ is contained in a rectangle $R$ that is mixed by
$\phi$. So, $\phi^{-j}(x) \not \leftrightarrow \phi^{-j+1}(x)$, and again aperiodicity of $x$ together with
Lemma~\ref{L:hopping sim} implies $\phi^{-j}(x) \not \sim x$, so $k_-(x) \geq -j+1$.

\end{proof}

We now prove that $k_\pm(x)$ depends only on the cell containing $x$ in its interior.

\begin{proposition}[Structure of $\sim$] \label{P:structure of sim}
Let $e$ be a cell of $Y$ that is not a singular point or a marked point.  If $x,y \in \intr(e)$, then $k_\pm(x) =
k_\pm(y)$.  If $x \in \intr(e)$ and $y \in \partial e$, then
\[k_-(y) \leq k_-(x) \leq k_+(x) \leq k_+(y).\]
In particular, we can define $k_\pm(e) = k_\pm(x)$ for any $x \in \intr(e)$ (independently of the choice of $x \in
\intr(e)$), and for every integer $\alpha \in [k_-(e),k_+(e)]$, the map $\phi^\alpha|_e$ is cellular with respect to $Y$.
\end{proposition}

\begin{proof}
We suppose first that $x \in \intr(e)$ and $y \in e$ and we prove that $k_+(x) \leq k_+(y)$.  If $k_+(x) = 0$, there is nothing to prove
so suppose $k = k_+(x) > 0$.  Combining Lemmas \ref{L:hopping sim} and \ref{L:orbit sim}, there exists a sequence of
rectangles $R_{i_0}, \dots, R_{i_{k-1}}$, so that
\[ (\phi^j(x),R_{i_j}) \leftrightarrow (\phi^{j+1}(x),\phi(R_{i_j}))\]
for $j = 0,\dots,k-1$.  For each $j$, Proposition \ref{P:sim1 and Y} implies that $\phi|_{R_{i_j}}$ is cellular with respect to $Y$.  Since $e$ is a
cell in $R_{i_0}$, by induction and Proposition \ref{P:sim1 and Y}, we see that $\phi^{j+1}(e)$ is a cell in both $\phi(R_{i_j})$ and $R_{i_{j+1}}$.
Hence $\phi^j(y) \in R_{i_j}$ and $\phi^{j+1}(y) \in \phi(R_{i_j})$, and therefore
\[ (\phi^j(y),R_{i_j}) \leftrightarrow (\phi^{j+1}(y),\phi(R_{i_j}))\]
for every $j = 0,\dots,k-1$.  It follows that $y \sim \phi^j(y)$ for $j = 1,\dots,k$ and thus $k_+(y) \geq k =
k_+(x)$.

Observe that if $y \in \intr(e)$, we can reverse the roles of $x$ and $y$ to obtain $k_+(y) \leq k_+(x)$, and hence
$k_+(x) = k_+(y)$.  Otherwise, $y \in \partial e$, and we only know $k_+(x) \leq k_+(y)$.

The proof that $k_-(x) = k_-(y)$ if $y \in \intr(e)$ and $k_-(y) \leq k_-(x)$ if $x \in \partial e$ follows a similar
argument.

That we can define $k_\pm(e) = k_\pm(x)$ for any $x \in \intr(e)$ now follows.  Finally, the fact that $\phi^\alpha|_e$
is cellular is proven in the course of the proof above.
\end{proof}

%%%%%%%%%%%%%%%%%%%%%%%%%%%%%%%
%%%%%%%%%%%%%%%%%%%%%%%%%%%%%%%
\subsection{The quotient of $S$}
%%%%%%%%%%%%%%%%%%%%%%%%%%%%%%%
%%%%%%%%%%%%%%%%%%%%%%%%%%%%%%%

Denote $S/\!\!\sim$ by $T$ and let $\pi:S \to T$ be the quotient map.

\begin{proposition} \label{P:W structure}
Let $(\phi:S \to S) \in \Psi_P$ and let $\mathcal R$ be a small Markov partition for $\phi$. There is a cell
structure $W = W(\phi,\mathcal R)$ on $T$, so that $\pi$ is cellular with respect to the cell structure $Y =
Y(\phi,\mathcal R)$ on $S$. Moreover, for each cell $e$ of $Y$, $\pi$ restricts to a homeomorphism from the interior of $e$ onto
the interior of a cell $\pi(e)$ of $W$.  Finally, there is a $D = D(P)$ so that $W$ has $D$--bounded complexity.
\end{proposition}

\begin{proof}
It follows from Proposition~\ref{P:structure of sim} that $\sim$ defines an equivalence relation on the cells  of
$Y$, which we also call $\sim$,  by declaring
\[ e \sim \phi^\alpha(e) \quad \textnormal{for all} \quad k_-(e) \leq \alpha \leq k_+(e). \]
If $e$ is a cell in $Y$ that is neither a singular point nor a marked point, and $e \sim e'$, then appealing to
Lemma~\ref{L:hopping sim}, Lemma~\ref{L:orbit sim}, and Proposition~\ref{P:structure of sim}, there is a unique integer
$\alpha(e,e') \in [k_-(e),k_+(e)]$ so that $\phi^{\alpha(e,e')}(e) = e'$.  Moreover, by Proposition \ref{P:structure of
sim}, $\phi^{\alpha(e,e')}|_e:e \to e'$ is cellular with respect to $Y$.

An exercise in CW--topology shows that $T$ admits a cell structure $W$ with one cell for each equivalence class of
cells in $Y$.  Moreover, the characteristic maps for these cells can be taken to be the characteristic maps for cells
of $Y$, composed with $\pi$.  Indeed, if $e \sim e'$ in $Y$, and $e$ and $e'$ are $p$--cells with $p \geq 1$, then we may
assume that the characteristic maps $\psi_e$ and $\psi_{e'}$ for $e$ and $e'$, respectively, are related by
\[ \phi^{\alpha(e,e')} \circ \psi_e = \psi_{e'}. \]
Of course, for $0$--cells, the characteristic maps are canonical.  It follows that $\pi$ is cellular and is a
homeomorphism on the interior of any cell.

\medskip

Since $\pi$ is cellular and is a homeomorphism on the interior of each cell, Lemma \ref{L:bounded complexity 2} implies
that each $2$--cell of $W$ has $D$--bounded complexity for any $D \geq D_2 = D_2(P)$.  Recall that the $0$-- and $1$--cells trivially
 have $D$--bounded complexity for all $D \geq 2$.
Lemma \ref{L:bounded complexity 2} also implies that every rectangle $R \in \mathcal R$ is subdivided into at most
$D_2$ $2$--cells of $Y$. By Corollary \ref{C:relation 1 prop} there are at most $E_Y$ $Y$--equivalence classes of
rectangles, and therefore there are at most $D_2E_Y$ $2$--cells of $W$.  Since each $2$--cell has $D_2$--bounded
complexity, there are at most $D_2^2E_Y$ $0$-- and $1$--cells.  Thus, $W$ has $D$--bounded complexity for $D =
D_2^2E_Y$.
\end{proof}

%%%%%%%%%%%%%%%%%%%%%%%%%%%%%%%%%%%%%%%%%%%%%%%%%%%%%%%%%%%%%%%%%%%%%%%%%%%%%%%%%%%%%%%%%
%%%%%%%%%%%%%%%%%%%%%%%%%%%%%%%%%%%%%%%%%%%%%%%%%%%%%%%%%%%%%%%%%%%%%%%%%%%%%%%%%%%%%%%%%
%%%%%%%%%%%%%%%%%%%%%%%%%%%%%%%%%%%%%%%%%%%%%%%%%%%%%%%%%%%%%%%%%%%%%%%%%%%%%%%%%%%%%%%%%
%%%%%%%%%%%%%%%%%%%%%%%%%%%%%%%%%%%%%%%%%%%%%%%%%%%%%%%%%%%%%%%%%%%%%%%%%%%%%%%%%%%%%%%%%
%%%%%%%%%%%%%%%%%%%%%%%%%%%%%%%%%%%%%%%%%%%%%%%%%%%%%%%%%%%%%%%%%%%%%%%%%%%%%%%%%%%%%%%%%

%%%%%%%%%%%%%%%%%%%%%%%%%%%%%%%%%%%%%%%%%%%%%%%%%%%%%%%%%%%%%
%%%%%%%%%%%%%%%%%%%%%%%%%%%%%%%%%%%%%%%%%%%%%%%%%%%%%%%%%%%%%
\section{Cell structure for the mapping torus} \label{S:3-mfd}
%%%%%%%%%%%%%%%%%%%%%%%%%%%%%%%%%%%%%%%%%%%%%%%%%%%%%%%%%%%%%
%%%%%%%%%%%%%%%%%%%%%%%%%%%%%%%%%%%%%%%%%%%%%%%%%%%%%%%%%%%%%

Let $\phi:S \to S$ be a pseudo-Anosov homeomorphism, and let $M = M_\phi$ denote the mapping torus.  This is the $3$--manifold obtained as
a quotient of $S \times [0,1]$ by identifying $(x,1)$ with $(\phi(x),0)$ for all $x \in S$.  We view $S$ as embedded in
$M$ as follows:
\[S \to S \times \{0 \} \to S \times [0,1] \to M.\]
Let $\{\phi_t \, | \, t \in \mathbb R\}$ denote the suspension flow on $M$; this is the flow on $M$ determined by the
local flow $\widetilde \phi_t(x,s) = (x,s+t)$ on $S \times [0,1]$. The time-one map restricted to $S$ is the first return map,
that is, $\phi_1|_S = \phi$. From this it follows that $\phi_k(x) = \phi^k(x)$ for every integer $k \in \mathbb Z$ and
$x \in S$.

We also have the punctured surface version.  The punctured surface $S^\circ$  embeds in $S$ and is $\phi$--invariant.
Hence, we may view $M^\circ$ as $M_{\phi|_{S^\circ}}$, the mapping torus of $\phi|_{{S^\circ}}:S^\circ \to S^\circ$, embedded in $M =
M_\phi$.

%%%%%%%%%%%%%%%%%%%%%%%%%%%%%%%%%%%%%%%%%%
\subsection{Boxes}
%%%%%%%%%%%%%%%%%%%%%%%%%%%%%%%%%%%%%%%%%%

%%%%%%%%%%%%%%%%  We don't need to know the equivalence classes are ordered in any particular way for this section.  %%%%%%%%

Let $(\phi:S \to S) \in \Psi_P$ and let $\mathcal R = \{R_i\}$ be a small Markov partition for $\phi$. For each
rectangle $R_i$, there is an associated \textit{box}, defined by
\[ B_i = \bigcup_{0 \leq t \leq 1} \phi_t(R_i).\]
More precisely, if $\rho_i:[0,1] \times [0,1] \to R_i$ is the parameterized rectangle, then the box is parameterized as
\[\hat \rho_i:([0,1] \times [0,1]) \times [0,1] \to B_i \]
where
$$\hat \rho_i({\bf u},t) = \phi_t(\rho_i({\bf u})).$$
As with rectangles, the map $\hat \rho_i$ is only an embedding when restricted to the interior.  And as in the case of
rectangles, we abuse notation, and refer to $B_i$ as a subset of $M$, although it is formally a map into $M$.
 
%%%%%%%%%%%%%%%%%%%%%%%%%%%%%%%%%%%
\subsection{The cell structure}
%%%%%%%%%%%%%%%%%%%%%%%%%%%%%%%%%%%

We impose on $M$ a cell structure $\hat Y = \hat Y(\phi,\mathcal R)$ defined by the cell structures $X$ and $Y$ on $S$
as follows. The $0$--cells of $\hat Y$ are the $0$--cells of $Y$ (recall that $S \subset M$). The $1$--cells of $\hat Y$ are the $1$--cells of $Y$, called \textit{surface
$1$--cells}, together with suspensions of $0$--cells of $X$:

\[ \bigcup_{0 \leq t \leq 1} \phi_t(v) \mbox{ for } v \in X^{(0)}. \]
Because $X^{(0)} \subset Y^{(0)}$ and since $\phi(X^{(0)}) \subset Y^{(0)}$, the boundary of each of these 1--cells is contained
in $\hat Y^{(0)} = Y^{(0)}$, as required.  We call these the \textit{suspension $1$--cells}.

The $2$--cells of $\hat Y$ are the $2$--cells of $Y$, called \textit{surface $2$--cells}, together with suspensions of
$1$--cells of $X$:
\[ \bigcup_{0 \leq t \leq 1} \phi_t(e) \mbox{ for } e \mbox{ a } 1\mbox{--cell of } X. \]
Observe that for each $1$--cell $e$ of $X$, both $e$ and $\phi(e)$ are $1$--subcomplexes of $Y$.  Furthermore, since
the boundary of $e$ is contained in $X^{(0)}$, we see that the boundary of the $2$--cell defined by the above
suspension is contained in $\hat Y^{(1)}$, as required.  We call these $2$--cells \textit{suspension $2$--cells}.

Finally, the $3$--cells of $\hat Y$ are the boxes.  All $3$--cells can be thought of as a \emph{suspension $3$--cells}, since they are suspensions of surface $2$--cells.

\begin{proposition}  \label{P:hat Y}
There exists a positive integer $K_1 = K_1(P)$ with the following property.  If $(\phi:S \to S) \in \Psi_P$ and if
$\mathcal R$ is a small Markov partition for $\phi$, then each cell of $\hat Y = \hat Y(\phi,\mathcal R)$ has
$K_1$--bounded complexity.
\end{proposition}
\begin{proof}

All $0$--cells and $1$--cells trivially have $K_1$--bounded complexity for all $K_1\geq 2$.

Since $S$ is embedded in $M$ as the subcomplex $Y \subset \hat Y$, it follows from part \ref{part5} of Lemma
\ref{L:bounded complexity 2} that the surface $2$--cells of $\hat Y$ have $K_1$--bounded complexity for any $K_1 \geq
D_2 = D_2(P)$.

If $e$ is a suspension $2$--cell obtained by suspending a $1$--cell $e_0$ of $X$, then the boundary of $e$ consists of
$2$ suspension $1$--cells, and by parts \ref{part1} and \ref{part4} of Lemma \ref{L:bounded complexity 2}, at most
$2D_2$ surface $1$--cells.  Therefore, these $2$--cells have $K_1$--bounded complexity for any $K_1 \geq 2D_2 + 2$.

The boundary of each $3$--cell is a union 6 rectangles: a ``bottom'' and a ``top'' (which are rectangles of $\mathcal R$ and $\phi(\mathcal R)$, respectively), and four ``suspension sides''.
The number of $0$--cells is just the number of $0$--cells in the top and bottom rectangles (since all vertices lie in
$S$).  Each of the top and bottom rectangles is a union of at most $D_2$ surface $2$--cells by parts \ref{part2} and \ref{part3} of
Lemma \ref{L:bounded complexity 2}.  Each surface $2$--cell has $D_2$--bounded complexity, so has at most $D_2$
vertices in its boundary.  It follows that each $3$--cell has at most $2D_2^2$ vertices in its boundary.

The number of suspension $2$--cells in the boundary of a $3$--cell $B_i$ is the number of $1$--cells of $X$ in the
boundary of $R_i \in \mathcal R$.  Since the cells of $X$ have $D_1 = D_1(P)$--bounded complexity by Lemma
\ref{L:bounded complexity 1}, and the $2$--cells of $X$ are precisely the rectangles of $\mathcal R$, it follows that
the boundary of a $3$--cell has at most $D_1$ suspension $2$--cells in its boundary.  Combining this with the bounds
from the previous paragraph on the number of surface $2$--cells, it follows that the boundary of a $3$--cell has at
most $D_1+2D_2$ $2$--cells in its boundary. Finally, since the boundary of a $3$--cell is a $2$--sphere, the Euler
characteristic tells us that the number of $1$--cells in the boundary is 2 less than the sum of the numbers of
$0$--cells and $2$--cells, and so is at most $D_1 + 2D_2 + 2D_2^2 \leq D_1 + 4D_2^2$.  It follows that if $K_1 \geq D_1
+ 4D_2^2$, then each $3$--cell has $K_1$--bounded complexity.

Therefore, setting $K_1 = D_1 + 4D_2^2$ completes the proof.
\end{proof}

A subset of the suspension $1$--cells are the \textit{singular} and \textit{marked $1$--cells}; these are the
$1$--cells of the form
\[ \bigcup_{0 \leq t \leq 1} \phi_t(v) \]
for $v$ either a singular point or marked point, respectively.  These $1$--cells, together with their vertices, form
$1$--dimensional subcomplexes called the \textit{singular} and \textit{marked subcomplexes}. These subcomplexes are
unions of circles in $M$, and $M^\circ$ is obtained from $M$ by removing them.

%%%%%%%%%%%%%%%%%%%%%%%%%%%%%%%%%%%%%%%%%%%%%%%%%%%%%%%%%%%%%%%%%%%%%%%%%%%%%%%%%%%%%%%%%
%%%%%%%%%%%%%%%%%%%%%%%%%%%%%%%%%%%%%%%%%%%%%%%%%%%%%%%%%%%%%%%%%%%%%%%%%%%%%%%%%%%%%%%%%
%%%%%%%%%%%%%%%%%%%%%%%%%%%%%%%%%%%%%%%%%%%%%%%%%%%%%%%%%%%%%%%%%%%%%%%%%%%%%%%%%%%%%%%%%
%%%%%%%%%%%%%%%%%%%%%%%%%%%%%%%%%%%%%%%%%%%%%%%%%%%

%%%%%%%%%%%%%%%%%%%%%%%%%%%%%%%%%%%%%%%%%%%%%%%%%%%%%%%%%%%%%%%%%%%%%%%%%%%%%%%%%%%%%%%%%
%%%%%%%%%%%%%%%%%%%%%%%%%%%%%%%%%%%%%%%%%%%%%%%%%%%%%%%%%%%%%%%%%%%%%%%%%%%%%%%%%%%%%%%%%
\section{Quotient spaces I: bounded complexity 3--complexes} \label{S:quotient I}
%%%%%%%%%%%%%%%%%%%%%%%%%%%%%%%%%%%%%%%%%%%%%%%%%%%%%%%%%%%%%%%%%%%%%%%%%%%%%%%%%%%%%%%%%
%%%%%%%%%%%%%%%%%%%%%%%%%%%%%%%%%%%%%%%%%%%%%%%%%%%%%%%%%%%%%%%%%%%%%%%%%%%%%%%%%%%%%%%%%

In this section we use the notion of $Y$--equivalence to produce a quotient of the compact $3$--manifold $M = M_\phi$;
the quotient we obtain will be compact, but might not be a manifold. We will prove that this quotient admits a
$3$--dimensional cell structure with uniformly bounded complexity for which the quotient map is cellular. The singular
and marked subcomplexes of $\hat Y$ will define subcomplexes of the quotient, and in the next section we will prove
that the space obtained by removing these subcomplexes from the quotient is homeomorphic to $M^\circ$, the corresponding mapping torus of
the punctured surface $S^\circ$. First, we define and analyze the quotients.

Let $(\phi:S \to S) \in \Psi_P$ be a pseudo-Anosov homeomorphism and let $\mathcal R = \{R_i\}_{i=1}^n$ be a small
Markov partition for $\phi$, indexed so that each $Y$--equivalence class has the form $\{R_i,\dots,R_{i+k}\}$ with
$R_{i+j} = \phi^j(R_i)$; see Section \ref{S:equiv rel}.  Let $\phi_t$ be the suspension flow on $M$. As in
Section~\ref{S:3-mfd}, we obtain collection of boxes $\{B_i\}_{i=1}^n$.  Let $\hat Y$ be the cell structure on $M$ as
described in Section \ref{S:3-mfd}.

%%%%%%%%%%%%%%%%%%%%%%%%%%%
\subsection{Prisms}
%%%%%%%%%%%%%%%%%%%%%%%%%%%

To each $Y$--equivalence class of rectangles, we will associate a \textit{prism} as follows.  If
$\{R_i,\dots,R_{i+k}\}$ is a $Y$--equivalence class, then the associated prism is
\[P_i = \bigcup_{0 \leq t \leq k} \phi_t(R_i). \]
The rectangles $R_i$ and $R_{i+k}$ are the \textit{bottom} and \textit{top} of the prism, respectively. Alternatively,
if $\{R_i,R_{i+1},\dots,R_{i+k}\}$ is a $Y$--equivalence class with more than one element, then the
associated prism is
\[ P_i = B_i \cup B_{i+1} \cup \cdots \cup B_{i+k-1}. \]
On the other hand if $\{R_i\}$ is a $Y$--equivalence class consisting of the single rectangle $R_i$, then $P_i = R_i
\subset S \subset M$. If a box is contained in a prism, then we will say it is a \textit{filled box}. See Figure
\ref{F:prism}.

Let $L$ be the union of the prisms in $M$. For clarification, we note that $L$ is the union of $S$ together with the
set of filled boxes, and is a subcomplex of $\hat Y$.

\begin{figure}[htb]
\centerline{}\centerline{}
\begin{center}
\ \psfig{file=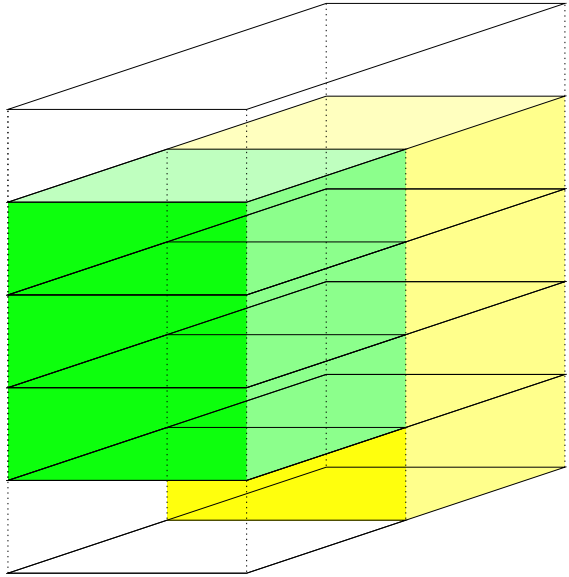,height=3truein} \caption{The figure shows two prisms (green and yellow) in $M_\phi$.  The
prisms stop near the top of the picture as the rectangles at the tops of the prisms are mixed: they are both mapped by $\phi$ into the topmost rectangle in the picture.  At the bottom of the picture, both
rectangles are unmixed, but only one of the boxes (the yellow one) is filled.  This means that the rectangle directly
below the green prism is not $Y$--equivalent to its image by $\phi$, the bottom of the green prism (although it is $h$--equivalent).}
\label{F:prism}
\end{center}
\end{figure}

The prisms define an equivalence relation $\approx$ on $M$ that is a ``continuous version'' of $\sim$ on $S$. More
precisely, we declare $y \approx z$ if and only if either $y = z$ or else there exists an interval $I = [0,r]$ or $I =
[r,0]$ in $\mathbb R$ so that $\phi_r(z) = y$ and
\[ \bigcup_{t \in I} \phi_t(z) \subset L.\]
That is, $y \approx z$ if there is an arc of a flow line containing $y$ and $z$ that is contained in $L$.

\bigskip

We will be particularly interested in the flow lines in $M$ through points  in $S$.  We denote the flow line
in $M$ through $x \in S$ by
\[ \ell_x = \bigcup_{t \in \mathbb R} \phi_t(x).\]

\begin{proposition} \label{P:prism relation}
The equivalence relation $\approx$ restricted to $S$ is precisely the equivalence relation $\sim$.
\end{proposition}

\begin{proof}

According to Lemma \ref{L:hopping sim}, if $x \sim x'$, then after reversing the roles of $x$ and $x'$, we have
rectangles $R_{i_0},\dots,R_{i_{k-1}} \in \mathcal R$ so that
\[ x \leftrightarrow \phi(x) \leftrightarrow \phi^2(x) \leftrightarrow \cdots \leftrightarrow \phi^{k-1}(x) \leftrightarrow \phi^k(x) = x'\]
with $(\phi^j(x),R_{i_j}) \leftrightarrow (\phi^{j+1}(x),\phi(R_{i_j}))$ for $j = 0,\dots,k-1$.  Therefore,
\[ \phi^j(x), \phi^{j+1}(x) \in \bigcup_{0 \leq t \leq 1} \phi_t(\phi^j(x)) = \ell_{\phi^j(x)} \cap B_{i_j}.\]
By the definition of the $\leftrightarrow$ relation, we have that $R_{i_j} \stackrel{Y}{\sim} \phi(R_{i_j})$, and so
$B_{i_j}$ is a filled box.  Thus $\phi^j(x) \approx \phi^{j+1}(x)$ for each $j = 0,\dots,k-1$.  By transitivity, $x
\approx x'$.

For the other direction, suppose $x \approx x'$.   We have that $x' \in \ell_x$ and the arc $\ell_x^0 \subset \ell_x$
from $x$ to $x'$ is contained in $L$.  Say that $x'=\phi^j(x)$ for $j \geq 0$ (reverse the roles of $x$ and $x'$ if
necessary).  We can write $\ell_x^0$ as
\[ \ell_x^0 = \bigcup_{0 \leq t \leq j} \phi_t(x) =  \bigcup_{i=0}^{j-1} \bigcup_{0 \leq t \leq 1} \phi_t(\phi^i(x)). \]
Since each arc
\[ \bigcup_{0 \leq t \leq 1} \phi_t(\phi^i(x)) \]
lies in $L$, it is contained in a filled box.  It follows that $\phi^i(x) \leftrightarrow \phi_1(\phi^i(x)) =
\phi^{i+1}(x)$ (recall that a $Y$--equivalence between rectangles induces $\leftrightarrow$--relations between all
pairs of corresponding points). Therefore
\[ x \leftrightarrow \phi(x) \leftrightarrow \phi^2(x) \leftrightarrow \, \cdots \, \leftrightarrow \phi^j(x) = x'. \]
Since $\sim$ is the transitive closure of $\leftrightarrow$, it follows that $x \sim x'$.
\end{proof}

%%%%%%%%%%%%%%%%%%%%%%%%%%%%
\subsection{Flowing out of $L$}
%%%%%%%%%%%%%%%%%%%%%%%%%%%%

The next proposition provides a bound to how long a flow line can stay in $L$.

\begin{proposition} \label{P:arc fibers}
Let $(\phi:S \to S) \in \Psi_P$ and let $\mathcal R$ be a small Markov partition for $\phi$.  If $x \in S$ is not a
singular point or a marked point, then $\ell_x \cap L$ is a union of compact arcs and points in
$\ell_x$. Moreover, there exists an integer $Q = Q(\phi,\mathcal R)$ so that each component of $\ell_x \cap L$ has
length strictly less than $Q$ (with respect to the flow parameter).
\end{proposition}

\begin{proof}

Fix $x \in S$ to be a nonsingular, unmarked point.  Since the set of such points is invariant under $\phi$, we have that $\ell_x
\cap S = \{\phi^j(x)\}_{j=-\infty}^\infty$ contains no singular or marked points.

We need only consider the component $\ell_x^0 \subset \ell_x \cap L$ containing $x$ (since every component has this
form for some $x$).  Recall that  Lemma \ref{L:orbit sim} implies that the $\sim$--equivalence class of $x$ is precisely
\[\{ \phi^{k_-(x)}(x), \dots, \phi^{k_+(x)}(x) \} \]
for integers $k_-(x) \leq 0 \leq k_+(x)$.  By Proposition~\ref{P:prism relation}, it follows
that $\ell_x^0$ is equal to the compact arc
\[ \bigcup_{k_-(x) \leq t \leq k_+(x)} \phi_t(x) = \ell_x^0 \]
of length $k_+(x) - k_-(x)$ (if $k_-(x) = 0 = k_+(x)$ the arc is a point).

On the other hand, Proposition~\ref{P:structure of sim} implies $k_\pm(x) = k_\pm(e)$, where $e$ is the unique cell containing $x$ in its interior.
Therefore, the length of the longest arc of intersection of a flow line not passing through a marked point or singular point is
\[ Q' = \max \left\{ k_+(e) - k_-(e) \, | \, e \mbox{ is an unmarked, nonsingular cell of }Y \right\}, \]
which is finite since there are only finitely many cells in $Y$.  Setting $Q=Q'+1$ completes the proof.
\end{proof}

%%%%%%%%%%%%%%%%%%%%%%%%%%%%%%
\subsection{Product structures}
%%%%%%%%%%%%%%%%%%%%%%%%%%%%%%

In the next proposition, we identify the box $B_i$ with $R_i \times [0,1]$ via the suspension flow as described in
Section \ref{S:3-mfd}.

Given a space $V$, a subspace $U \subset V$ and a cell structure $Z$ on $V$, if $U$ happens to be a subcomplex with
respect to the cell structure $Z$, then we will refer to the induced cell structure on $U$ as the
\textit{restriction} of $Z$ to $U$, and write it as $Z|U$.

\begin{proposition} \label{P:product cells}
For each filled box $B_i$, the restriction $\hat Y|{B_i}$ agrees with the product cell structure on $R_i \times [0,1]$
coming from $Y|{R_i}$ and the cell structure on $[0,1]$ that has a single $1$--cell.
\end{proposition}

\begin{proof}

In a product cell structure, the cells all have the form $e_0 \times e_1$, where $e_0$ and $e_1$ are cells of the first
and second factors, respectively.  Thus, the cells of the product $Y|R_i \times [0,1]$ are of the form (1)
$e_0 \times [0,1]$, (2) $e_0 \times \{0\}$, and (3) $e_0 \times \{1\}$.   So the proposition is saying that for each
cell $e$ of $\hat Y$ in a filled box $B_i$, there is a cell $e_0$ from $Y|{R_i}$ so that one of the following holds:
\[ (1) \, \, \,  e = \displaystyle \bigcup_{0 \leq t \leq 1} \phi_t(e_0), \quad \quad (2) \, \, \, e = e_0, \quad \quad \mbox{ or } \quad \quad (3) \, \, \, e = \phi_1(e_0).\]

So suppose
\[ B_i = \bigcup_{0 \leq t \leq 1} \phi_t(R_i) \]
is a filled box.  Then $R_i \stackrel{Y}{\sim} R_{i+1} = \phi(R_i)$, and, according to Proposition~\ref{P:sim1 and Y},
we have that $\phi|_{R_i}:R_i \to  R_{i+1}$ is cellular with respect to $Y|{R_i} = X|{R_i}$ and $Y|{R_{i+1}} =
X|{R_{i+1}}$.

Since the suspension cells are precisely the suspensions of cells of $X$, which in $R_i$ are exactly the cells of $Y$,
it follows that all suspension cells are of the form (1) above.  All other cells are surface cells in $R_i$ and
$R_{i+1}$. The cells in $R_i$ are of course of the form (2).  Since, again, $\phi|_{R_i}$ is cellular with respect to $Y$, every cell of $Y|{R_{i+1}}$ is of the form $\phi_1(e_0)$ for some cell $e_0$ of $Y|{R_i}$, and thus has the form (3).
\end{proof}

%%%%%%%%%%%%%%%%%%%%%%%%%%%%%%%%%%%%%%%%%%
\subsection{Cell structure on the quotient}
%%%%%%%%%%%%%%%%%%%%%%%%%%%%%%%%%%%%%%%%%%

Let $N$ denote $M/\!\!\approx$ and let $p:M \to N$ be the quotient map.  According to Proposition~\ref{P:prism
relation}, the restriction to $S$ of the relation $\approx$ is precisely the relation $\sim$ and so the inclusion of $S \to M$ descends to an
inclusion $T \to N$, and we have the following commutative diagram:
\[
\xymatrix{ S \ar[r]^\pi\ar[d] & T \ar[d] \\
M \ar[r]^p & N\\}
\]

\begin{proposition} \label{P:hat W bounded}
The quotient $N$ admits a cell structure $\hat W = \hat W(\phi,\mathcal R)$ so that $p$ is cellular with respect to
$\hat Y$ on $M$. Moreover, there exists $K = K(P)$ so that $\hat W$ has $K$--bounded complexity.
\end{proposition}

\begin{proof}

We view $T$ as a subset of $N$.  First, recall that $L$ is the union of the prisms and that $S$ is contained in $L$, and observe that all of $L$ is mapped (surjectively)
to $T$.  Indeed, on any filled box $B_i$ the map to $N$ is obtained by first projecting onto $R_i$ using its product
structure coming from the flow, then projecting $R_i$ to $T$ by $\pi$.   By Proposition \ref{P:product
cells}, this is a cellular map from $\hat Y|L$ to $W$, the cell structure on $T$ from Proposition~\ref{P:W structure}.

Let $B_{i_1},\dots,B_{i_k}$ denote the set of unfilled boxes.  The $3$--manifold $M$, with cell structure $\hat Y$, is the union of two subcomplexes:
\[L \qquad \textnormal{and} \qquad L^c = S \cup B_{i_1} \cup \cdots \cup B_{i_k}. \]
Note that $L$ and $L^c$ are not complementary, as both contain $S$.  Since both $S$ and $L$ map onto $T$, it follows that $L^c$ maps onto $N$:
\[ N = p(L^c) = T \cup p(B_{i_1}) \cup \cdots \cup p(B_{i_k}). \]

We now construct a cell structure $\hat W$ on $N$ for which $p|_{L^c}:L^c \to N$ is cellular with respect to $\hat
Y|{L^c}$ and for which $T$ is a subcomplex with $\hat W|T = W$.  If we can do this, then since $p|_L:L \to N$ is
already cellular, and $M$ is the union of the two subcomplexes, it will follow that $p$ is cellular.

The above description of $N$ as the image of $S$ and the unfilled boxes indicates the way to build the cell structure $\hat W$ on $N$.  Namely, we start by giving $T$ the cell structure $W$ and then describe the remaining cells as images of certain cells in $L^c$.

First suppose that $e$ is a suspension $1$--cell that is not contained in $L$.  Then since any suspension $1$--cell is contained in the intersection of some set of boxes, it must be that the $1$--cell meets $L$ only in its endpoints.  In particular, $p$ is injective when restricted to the interior of $e$, and the endpoints map into $W^{(0)}$.  We take $p(e)$ as a $1$--cell in $\hat W$.

In addition to the $2$--cells of $W$, we also construct a new $2$--cell $e'$ for each suspension $2$--cell $e$ that is
not contained in $L$, as follows.  The suspension $2$--cells are contained in intersections of boxes, and so since $e$
is not contained in $L$, it meets $L$ only in its boundary.  Therefore, $p$ is injective on the interior of $e$.
Moreover, the boundary contains two suspension $1$--cells and two arcs that are unions of surface $1$--cells of $\hat
Y$.   The map $p$ sends all the surface $1$--cells of $\partial e$ cellularly into $T$, and is injective on the
interior of each cell.  Each of the two suspension $1$--cells is either sent to a $1$--cell constructed in the previous
section (hence injectively on the interior), or is collapsed to a $0$--cell.  The cell $e'$ is essentially $p(e)$, the
only difference being the cell structure on $\partial e'$ has collapsed those suspension $1$--cells in $\partial e$
that are collapsed by $p$ in $N$.

Finally, the $3$--cells come from the unfilled boxes $B_{i_1},\dots,B_{i_k}$.  Each unfilled box has a graph on its
boundary coming from the cell structure $\hat Y$.  Moreover, by Proposition \ref{P:hat Y}, this graph has at most $K_1
= K_1(P)$ vertices, edges, and complementary regions, all of which are disks.  These disks are suspension and surface
$2$--cells in $\hat Y|{L^c}$.  The $3$--cells in $\hat W$ are obtained by first collapsing any suspension $1$--cells
and $2$--cells in the boundary of the box (using the product structure from the flow) that are contained in
$L$, then projecting by $p$ to $N$.  This collapsing produces a new graph in the boundary of the $3$--cell, but the new
graph has no more edges, vertices or complementary regions than the original one.  Moreover, all complementary
components are still disks; see Figure \ref{F:3cellcollapse} for a picture of the $3$--cell and the new cell structure
on the boundary. With this new cell structure on the boundary, the restriction of $p$ to the interior of each cell is
injective.  We also observe that each of these $3$--cells has $K_1$--bounded complexity.

\begin{figure}[htb]
\begin{center}
\ \psfig{file=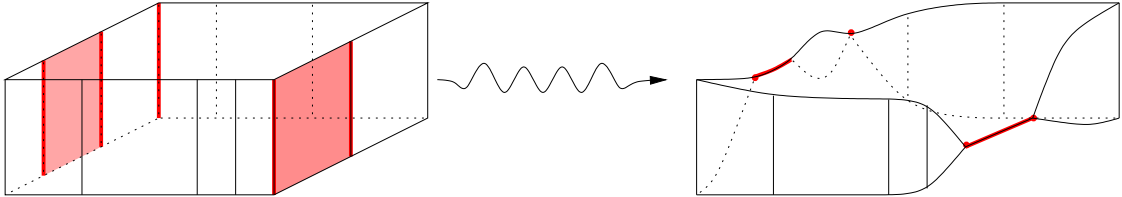,height=1.05truein} \caption{On the left is a box in $M$ with the suspension cells
contained in $L$ shown in red.  On the right is a picture of the $3$--cell after collapsing the suspension cells of the
boundary contained in $L$.} \label{F:3cellcollapse}
\end{center}
\end{figure}

It is straightforward to check that this is the desired cell structure; $p$ is cellular on $L^c$,
and by construction $T$ is a subcomplex with $\hat W|T = W$.  Therefore, $p$ is cellular on all of $N = L \cup L^c$ as
discussed above, and we have proven the first part of the proposition.

\medskip

Now we find a $K = K(P) > 0$ so that $\hat W$ has $K$--bounded complexity.  First, the number of $3$--cells is precisely the
number of unfilled boxes.  Each unfilled box is given by
\[ \bigcup_{0 \leq t \leq 1} \phi_t(R) \]
where $R$ is a terminal rectangle in its $Y$--equivalence class.  Therefore, the unfilled boxes correspond
precisely to the $Y$--equivalence classes of rectangles.  By Corollary \ref{C:relation 1 prop}, there
are at most $E_Y = E_Y(P)$ unfilled boxes, and so at most $E_Y$ $3$--cells.  As mentioned above, each of these
$3$--cells has $K_1$--bounded complexity.

The cell structure $W$ has $D = D(P)$--bounded complexity by Proposition \ref{P:W structure}, so there are at most $D$ $1$--cells in $W$.  Each $1$--cell of $\hat W$
not in $W$ is contained in the boundary of some $3$--cell.  Since these have $K_1$--bounded complexity, and since there are at most $E_Y$ of these, it follows that there are at most $E_YK_1$ $1$--cells in $\hat W$ that are not in $W$.  Therefore, there are at most $(D+E_YK_1)$ $1$--cells in $\hat W$.

Similarly, the number of $2$--cells is bounded by the number of $2$--cells in $W$ plus the number of $2$--cells in each
of the unfilled boxes.  A count as in the previous paragraph implies that there are at most $D + E_YK_1$ $2$--cells in $\hat W$.  Finally,
the $2$--cells of $W$ have $D$--bounded complexity by Proposition \ref{P:W structure}, and each of the suspension
$2$--cells has $K_1$--bounded complexity since it is in the boundary of a $3$--cell, which has $K_1$--bounded complexity.

Therefore, setting
\[ K = K(P) = D + E_Y K_1,\]
it follows that $\hat W$ has $K$--bounded complexity.
\end{proof}

Let
\[ \overline{\T(\Psi_P)} = \{ \hat W(\phi,\mathcal R) \,  | \, \phi \in \Psi_P \mbox{ and } \mathcal R \mbox{ a small Markov partition for }\phi \}/\textnormal{CW--homeomorphism}.\]
The following is immediate from Propositions \ref{P:bounded complexity finite} and \ref{P:hat W bounded}.

\begin{corollary} \label{C:compact finiteness}
The set $\overline{\T(\Psi_P)}$ is finite.
\end{corollary}

%%%%%%%%%%%%%%%%%%%%%%%%%%%%%%%%%%%%%%%%%%%%%%%%%%%%%%%%%%%%%%%%%%%%%%%%%%%%%%%%%%%%%%%%%
%%%%%%%%%%%%%%%%%%%%%%%%%%%%%%%%%%%%%%%%%%%%%%%%%%%%%%%%%%%%%%%%%%%%%%%%%%%%%%%%%%%%%%%%%
%%%%%%%%%%%%%%%%%%%%%%%%%%%%%%%%%%%%%%%%%%%%%%%%%%%%%%%%%%%%%%%%%%%%%%%%%%%%%%%%%%%%%%%%%
%%%%%%%%%%%%%%%%%%%%%%%%%%%%%%%%%%%%%%%%%%%%%%%%%%%%%%%%%%%%%%%%%%%%%%%%%%%%%%%%%%%%%%%%%
%%%%%%%%%%%%%%%%%%%%%%%%%%%%%%%%%%%%%%%%%%%%%%%%%%%

%%%%%%%%%%%%%%%%%%%%%%%%%%%%%%%%%%%%%%%%%%%%%%%%%%%%%%%%%%%%%%%%%%%%%%%%%%%%%%%%%%%%%%%%%
%%%%%%%%%%%%%%%%%%%%%%%%%%%%%%%%%%%%%%%%%%%%%%%%%%%%%%%%%%%%%%%%%%%%%%%%%%%%%%%%%%%%%%%%%
\section{Quotient spaces II: Finitely many 3--manifolds} \label{S:quotient II}
%%%%%%%%%%%%%%%%%%%%%%%%%%%%%%%%%%%%%%%%%%%%%%%%%%%%%%%%%%%%%%%%%%%%%%%%%%%%%%%%%%%%%%%%%
%%%%%%%%%%%%%%%%%%%%%%%%%%%%%%%%%%%%%%%%%%%%%%%%%%%%%%%%%%%%%%%%%%%%%%%%%%%%%%%%%%%%%%%%%
We continue with the notation from the previous section: $M_{\phi|_{S^\circ}} = M^\circ \subset M = M_\phi$, $L$ is the
union of $S$ and the prisms, and $p:M \to N$ is the quotient defined by collapsing arcs of flow lines in $L$ to points.
Let $N^\circ$ denote $p(M^\circ)$, the complement in $N$ of the singular and marked subcomplexes.  We also write
$p:M^\circ \to N^\circ$ for the restriction. The main goal of this section is to prove the following.

\begin{theorem} \label{T:quotient htpc homeo}
The map $p:M^\circ \to N^\circ$ is homotopic to a homeomorphism.
\end{theorem}

Recall that a 3--manifold is \emph{irreducible} if every 2--sphere in the 3--manifold bounds a 3--ball.  We will deduce Theorem~\ref{T:quotient htpc homeo} from
Waldhausen's Theorem (the version we will use can be found in \cite[Corollary 13.7]{He}):

\begin{theorem}[Waldhausen]
\label{T:wald}
Suppose $M_1$ and $M_2$ are compact, orientable irreducible 3--manifolds with nonempty boundary.  If $f:(M_1,\partial M_1) \to (M_2,\partial M_2)$ is a map such that $f_*:\pi_1(M_1) \to \pi_1(M_2)$ is an isomorphism and $f_*:\pi_1(\partial_0 M_1) \to \pi_1(\partial_0 M_2)$ is an injection for each component $\partial_0 M_1 \subset \partial M_1$ and the component $\partial_0 M_2 \subset \partial M_2$ containing $f(\partial_0 M_1)$, then $f$ is homotopic to a homeomorphism.
\end{theorem}

Assuming Theorem~\ref{T:quotient htpc homeo}, we can prove Theorem \ref{theorem:main1}, which states that the set
$\T(\Psi_P^\circ)$ is finite.

\begin{proof}[Proof of Theorem \ref{theorem:main1}]

By Corollary \ref{C:compact finiteness}, $\overline{\T(\Psi_P)}$ is a finite set of CW--homeomorphism types of cell complexes.
Therefore, since each $N^\circ$ is obtained by removing a $1$--subcomplex from a compact $3$--complex $N \in \overline{\T(\Psi_P)}$, it follows that the set
\[  \{N^\circ \, | \, N \textnormal{ represents an element of } \overline{\T(\Psi_P)} \} /\, \textnormal{homeomorphism} \]
is finite.  But by Theorem~\ref{T:quotient htpc homeo}, this set is equal to $\T(\Psi_P^\circ)$.
\end{proof}

\medskip

So, we are left to prove Theorem~\ref{T:quotient htpc homeo}.  This will occupy the remainder of the section.

\bigskip

Let $\rho:M^\circ \to S^1$ denote the fibration coming from the mapping torus description of $M^\circ$. Let $\omega \in
\Omega^1(S^1)$ be a volume form so that $\int_{S^1} \omega = 1$, and let $\mu = \rho^*(\omega) \in \Omega^1(M^\circ)$
be the pullback (which on $S^\circ \times [0,1]$ we can take to be $dt$).  This $1$--form $\mu$ represents
the Poincar\'e dual of the class represented by $S^\circ \subset M^\circ$. Integration of $\mu$ defines an epimorphism
$\int\!\mu: \pi_1(M^\circ) \to \mathbb Z$, and we let
\[q:\widetilde M^\circ \to M^\circ\]
denote the infinite cyclic cover of $M^\circ$ determined by $\int\! \mu$.

Alternatively, $q$ is simply the cover corresponding to $\pi_1(S^\circ)$, and so is obtained by unwrapping the bundle.
Therefore, we have
\[ \widetilde M^\circ \cong S^\circ \times \mathbb R.\]
We make this diffeomorphism explicit as follows. First,
choose a lift $S^\circ \subset \widetilde M^\circ$ of the embedding $S^\circ \subset M^\circ$.  The flow $\phi_t$ lifts to a flow
$\widetilde \phi_t$ on $\widetilde M^\circ$.  We now define a map
\[ S^\circ \times \mathbb R \to \widetilde M^\circ \]
by
\[ (x,t) \mapsto \widetilde \phi_t(x).\]
We use this diffeomorphism to identify $\widetilde M^\circ$ with $S^\circ \times \mathbb R$.

\medskip

The covering group for $q:\widetilde M^\circ \to M^\circ$ is infinite cyclic.  Let $\delta$ be the generator that induces a translation by $+1$ in the second factor of $\widetilde M^\circ \cong S^\circ \times \R$.  With respect to the product structure, we have
\[ \delta(x,t) = (\phi^{-1}(x),t+1). \]
Writing the last formula in terms of the flow $\widetilde \phi_t$, we obtain
\[\delta(\widetilde \phi_t(x)) = \delta(x,t) =
(\phi^{-1}(x),t+1) =\widetilde \phi_{t+1}(\phi^{-1}(x)), \] and iterating, we obtain
\begin{equation} \label{eqn:flow equivariance}
\delta^k(\widetilde \phi_t(x)) = \widetilde \phi_{t+k}(\phi^{-k}(x)).
\end{equation}

\bigskip

We will study the quotient $N^\circ$ of $M^\circ$ by looking at the associated quotient of $\widetilde M^\circ$.  Let  $L^\circ =L \cap
M^\circ$, and let $\widetilde L^\circ = q^{-1}(L^\circ)$.  For any $x \in S^\circ$, we consider the flow line
\[ \widetilde \ell_x = \bigcup_{t \in \mathbb R} \widetilde \phi_t(x).\]
Since $\widetilde \ell_x$ is a lift of the flow line $\ell_x$, Proposition
\ref{P:arc fibers} implies that $\widetilde \ell_x \cap \widetilde L^\circ$ is a union of compact arcs (and points) of length less than
$Q = Q(\phi,\mathcal R)$ (measured with respect to the flow parameter). For $y,z \in \widetilde M^\circ$ define $y \equiv z$ if
there exists $x \in S^\circ$ so that $y,z$ are in the same component of $\widetilde \ell_x \cap \widetilde L^\circ$. Set
\[ \widetilde N^\circ = \widetilde M^\circ/\!\equiv \]
and let
\[ \widetilde p: \widetilde M^\circ \to \widetilde N^\circ \]
be the quotient map.

\bigskip

We will now give an alternative description of $\widetilde N^\circ$ as an open subset of $S^\circ \times \mathbb R$; we
will realize it as the image of a smooth map $f:\widetilde M^\circ \to S^\circ \times \mathbb R$.
Let $h \in C^\infty(M^\circ)$ be a smooth function that is positive on $M^\circ - L^\circ$ and is identically zero on
$L^\circ$. Let $\nu = h \mu$, and $\widetilde \nu = q^*(\nu)$. Since $\widetilde \nu$ is a pullback, it is invariant by
$\delta$:
\[\delta^*(\widetilde \nu) = \widetilde \nu. \]
Let $dt$ denote the 1--form on $\widetilde M^\circ \cong S^\circ \times \R$ obtained by pulling back the volume form on $\R$, and
observe that on $S^\circ \times \mathbb R$ we have
\[ \widetilde \nu = (h \circ q)\, dt.\]

Now define $f: S^\circ \times \mathbb R \to S^\circ \times \mathbb R$ by
\[ f(x,t) = \left( x, \int_0^t \widetilde \nu_x \right) \]
where the notation in the integral means the integral of $\widetilde \nu$ over the path $\phi_s(x)$ as $s$ runs from $0$ to $t$.
Since $\widetilde \nu$ is smooth, the map $f$ is smooth.  Set
\[ V = f(\widetilde M^\circ). \]
We remark that $V$ may or may not be all of $S^\circ \times \R$.

\begin{lemma}
\label{L:same fibers}
The fibers of $f$ are precisely the $\equiv$--equivalence classes. That is, $f$ and $\widetilde p$ have the same fibers.
\end{lemma}

\begin{proof}

First observe that if $x \neq x'$, then $f(\widetilde \ell_x) \cap f(\widetilde \ell_{x'}) = \emptyset$ since $f(\widetilde \ell_x)
= f(\{x \} \times \mathbb R) \subset \{x \} \times \mathbb R$.  It follows that  we can only have $y \equiv
z$ if $y$ and $z$ both lie on $\widetilde \ell_x$ for some $x$.  Therefore, it suffices to show that for any $x \in S^\circ$ and for all
$y,z \in \widetilde \ell_x$, we have $f(y) = f(z)$ if and only if $y \equiv z$.

To prove this, we observe that for any flow line $\widetilde \ell_x$, the restriction $\widetilde \nu_x$ of $\widetilde \nu$ to
$\widetilde \ell_x$ is zero precisely on the intersection $\widetilde \ell_x \cap \widetilde L^\circ$.  Therefore, $f|_{\widetilde \ell_x}$
is constant on each component of $\widetilde \ell_x \cap L^\circ$ and monotone (increasing) on the complementary intervals.  It
follows that the fibers of $f|_{\widetilde \ell_x}$ are precisely the $\equiv$--equivalence classes.
\end{proof}

Because $\widetilde \nu$  is invariant by $\delta$, the map $f$ semiconjugates $\delta$ to a map $\hat \delta: V \to V$.  That is,
$\hat \delta \circ f = f \circ \delta$.  An explicit formula for $\hat \delta$ is given by
\[ \hat \delta(x,s) = \left( \phi^{-1}(x), \ s + \int_0^1 \widetilde \nu_{\phi^{-1}(x)} \right).\]

\begin{proposition}
\label{P:covering V}
The image $V=f(S^\circ \times \mathbb R)$ is an open subset of
$S^\circ \times \mathbb R$, and is homeomorphic to $S^\circ \times \mathbb R$. Moreover, $\langle
\hat \delta \rangle$ acts properly discontinuously and freely on $V$, and with respect to this product structure,
$\hat \delta$ is given by $\phi^{-1}$ on the first factor.
\end{proposition}

\begin{proof}

Let $Q = Q(\phi,\mathcal R)$ be the constant from Proposition~\ref{P:arc fibers}.  As mentioned above (immediately after the definition of $\widetilde \ell_x$), the integer $Q$ is greater than the length of the component of any flow
line $\widetilde \ell_x^0 \subset \widetilde \ell_x \cap \widetilde L^\circ$.  In particular, it
follows that
\[ \int_0^{Q} \widetilde \nu_x > 0 \]
for all $x \in S$.

It follows that $f(S^\circ \times \{0\})  \cap f(S^\circ \times \{Q\}) = \emptyset$.  On the other hand, we have
\[\hat \delta^{Q}(f(S^\circ \times \{0\})) = f (\delta^{Q}(S^\circ \times \{0 \})) = f (S^\circ \times \{Q\}). \]
Since $f(S^\circ) = S^\circ = S^\circ \times \{0\} \subset V$, this implies
\[ S^\circ \cap \hat \delta^{Q}(S^\circ)= \emptyset \]
in $V$.

The region between $S^\circ$ and $\hat \delta^{Q}(S^\circ)$ in $S^\circ \times \R$ is homeomorphic to a product.  Indeed, it is exactly the region
between the graphs of the zero function on $S^\circ$ and the smooth positive function
\[x \mapsto \int_0^{Q} \widetilde \nu_x.\]
We denote this region by $\Delta$.  Observe that
\[ \Delta = \left\{ \left( x, \int_0^t \widetilde \nu_x \right) \, \Big| \, x \in S^\circ \mbox{ and } 0 \leq t \leq Q
\right\} = f(S^\circ \times [0,Q]) \]
and so $\Delta \subset V$.

The region $\Delta$ is a fundamental domain for the action of $\langle \hat \delta^Q \rangle$ on $V$, and $V$ is equal to
\begin{equation} \label{eqn:Delta translates}
\bigcup_{j=-\infty}^\infty \hat \delta^{jQ}(\Delta)
\end{equation}
with $\langle \hat \delta^Q \rangle$ acting properly discontinuously.
Since $\langle \hat \delta^Q \rangle$ is contained in $\langle \hat \delta \rangle$ with finite index, it follows
that $\langle \hat \delta \rangle$ also acts properly discontinuously. Because $\langle \hat \delta \rangle$ is torsion
free, its action on $V$ is also free.

Any homeomorphism from $\Delta$ to $S^\circ
\times [0,Q]$ that is the identity on the surface factor uniquely extends to a homeomorphism $V \to S^\circ \times \R$
that conjugates the action of $\langle \hat \delta^Q \rangle$ to the action of $\langle \delta^Q \rangle$.  As the homeomorphism is the identity on the first factor, the conjugate of $\hat \delta$ (which is different from $\delta$) restricts to $\phi^{-1}$ on the first factor.

From \eqref{eqn:Delta translates}, we also see that $V$ is an open set: any point is either in the interior of a
$\langle \hat \delta^Q \rangle$--translate of $\Delta$, or the union of two consecutive translates.
\end{proof}

Finally, we prove Theorem~\ref{T:quotient htpc homeo}, which states that the map $p:M^\circ \to N^\circ$ is homotopic to a homeomorphism.

\begin{proof}[Proof of Theorem \ref{T:quotient htpc homeo}]

We are going to make connections between the various spaces and maps we have constructed (and some we have yet to construct).  Throughout, the following diagram will serve as a guide.

\begin{center}
\label{F:big diagram}
\scalebox{1}{
\xymatrix{
& V \ar@{->}[r]^{\!\!\!\!\!\!\!\!\!\cong} \ar@{->}[d]^{\eta}& S^\circ \times \R \\
\widetilde{M}^\circ \ar@{->}[r]^{\widetilde p} \ar@{->}[d]_{q} \ar@{->}[ur]^{f}& \widetilde N^\circ \ar@{->}[dr]^{r}& \\
M^\circ \ar@{->}[r]^{p} & N^\circ \ar@{->}[r]^{\!\!\!\!\!\!\!\!\!\cong} & \widetilde N^\circ/\langle \hat \delta \rangle
}
}
\end{center}

The map $f: \widetilde M^\circ \to V \cong S^\circ \times \mathbb R$ is a proper continuous surjection, and hence is a quotient
map.  Also, by definition, $\widetilde p: \widetilde M^\circ \to \widetilde N^\circ$ is a quotient map.  Since the fibers of these two maps are the same  (Lemma~\ref{L:same fibers}), it follows that the quotients are homeomorphic.  Let $\eta:V \to \widetilde N^\circ$ be the homeomorphism for which $\eta \circ f = \widetilde p$.

We use the homeomorphism $\eta$ and the homeomorphism $V \cong S^\circ \times \mathbb R$ from Proposition~\ref{P:covering V} to identify $\widetilde N^\circ$ with $S^\circ \times \mathbb R$.  Conjugating by
this homeomorphism we obtain an action of $\langle \hat \delta \rangle$ on $\widetilde N^\circ$.  We abuse notation and simply
refer to the conjugate homeomorphism by the same name:
\[ \hat \delta:\widetilde N^\circ \to \widetilde N^\circ. \]

Now, $N^\circ$ is the image of $M^\circ$ via the quotient map $p:M^\circ \to N^\circ$, and since the covering map $q:\widetilde M^\circ \to
M^\circ$ is also a quotient map, the composition $p \circ q:\widetilde M^\circ \to N^\circ$ is a quotient map.  We also have the
covering map
\[r: \widetilde N^\circ \to \widetilde N^\circ/\langle \hat \delta \rangle \]
which is a quotient map, and hence so is $r \circ \widetilde p$.  Since the fibers of $p \circ q$ are the same as those of
$r \circ \widetilde p$, it follows that
\[N^\circ \cong \widetilde N^\circ/\langle \hat \delta \rangle.\]

General covering space theory implies that $\pi_1(N^\circ)$ fits into a short exact sequence:
\[ 1 \to \pi_1(\widetilde N^\circ) \to \pi_1(N^\circ) \to \mathbb Z \to 1. \]
The monodromy $\mathbb Z \to \Out(\pi_1(\widetilde N^\circ))$ is given by $1 \mapsto \hat \delta_*$.  On
the other hand, the isomorphism $\pi_1(\widetilde N^\circ) \to \pi_1(S^\circ)$ induced by projecting onto the first factor,
conjugates $\hat \delta_*$ to $\phi^{-1}_*$ by Proposition~\ref{P:covering V}. Therefore, $\pi_1(N^\circ)$ is an extension of $\mathbb Z$ by $\pi_1(S^\circ)$ with monodromy $\Z \to \Out(\pi_1(S^\circ))$ given by
$1 \mapsto \phi_*$. The fundamental group $\pi_1(M^\circ)$ also has such a description.  Since $\mathbb Z$ is free, each
sequence splits, and it follows that $\pi_1(M^\circ)$ is isomorphic to $\pi_1(N^\circ)$.
Moreover, the map $p:M^\circ \to N^\circ$ induces this isomorphism
\[p_*:\pi_1(M^\circ) \to \pi_1(N^\circ) \]
as can be seen in the lift $\widetilde p:\widetilde M^\circ \to \widetilde N^\circ$.

Since $N^\circ$ is covered by $S^\circ \times \mathbb R$ and $S^\circ$ is a hyperbolic surface, it follows that the universal
cover of $N^\circ$ is $\mathbb R^3$.  Therefore $N^\circ$ is irreducible.  Furthermore, $N^\circ$ is the complement of a
$1$--subcomplex of the cell complex $N$, which has $K$--bounded complexity by Proposition \ref{P:hat W bounded}.  Since $N$ can be subdivided into a simplicial complex (see the proof of Proposition \ref{P:bounded complexity finite}), it follows that $N^\circ$ is tame, that is, it is homeomorphic to the interior of a compact 3--manifold with boundary.

Because $N^\circ$ is tame, we can remove a product neighborhood of the ends to produce a \textit{compact core} $\widehat N^\circ$ for $N^\circ$, which is a compact submanifold for which the inclusion is a homotopy equivalence.  Since $p$ is a proper map, $p^{-1}(\widehat N^\circ)$ is a compact subset of $M^\circ$.  On the other hand, $M^\circ$ is also tame, and so we can also remove product neighborhoods of the ends of $M^\circ$ to arrive at a compact core $\widehat M^\circ$ for $M^\circ$ that contains $p^{-1}(\widehat N^\circ)$.

Using the product structure on the complement of $\widehat M^\circ$ in $M^\circ$ and of $\widehat N^\circ$ in
$N^\circ$, there are strong deformation retractions $h_M:M^\circ \to \widehat M^\circ$ and $h_N:N^\circ \to \widehat
N^\circ$. The map $F = h_N \circ p \circ h_M$ is therefore homotopic to $p$, and moreover the restriction $F|_{\widehat
M^\circ}:\widehat M^\circ \to \widehat N^\circ$, satisfies the hypotheses of Waldhausen's Theorem
(Theorem~\ref{T:wald}): the only thing to verify is that the boundary subgroups are mapped injectively, but that is
clear from the construction of $p$.  Therefore, $F|_{\widehat M^\circ}$ is homotopic to a homeomorphism.  Using this
homotopy and the product neighborhoods of the ends one can construct a homotopy from $F$, and hence from $p$, to a
homeomorphism.
\end{proof}

%%%%%%%%%%%%%%%%%%%%%%%%%%%%%%%%%%%%%%%%%%%%%%%%%%%%%%%%%%%%%%%%%%%%%%%%%%%%%%%%%%%%%%%%%
%%%%%%%%%%%%%%%%%%%%%%%%%%%%%%%%%%%%%%%%%%%%%%%%%%%%%%%%%%%%%%%%%%%%%%%%%%%%%%%%%%%%%%%%%
%%%%%%%%%%%%%%%%%%%%%%%%%%%%%%%%%%%%%%%%%%%%%%%%%%%%%%%%%%%%%%%%%%%%%%%%%%%%%%%%%%%%%%%%%
%%%%%%%%%%%%%%%%%%%%%%%%%%%%%%%%%%%%%%%%%%%%%%%%%%%%%%%%%%%%%%%%%%%%%%%%%%%%%%%%%%%%%%%%%
%%%%%%%%%%%%%%%%%%%%%%%%%%%%%%%%%%%%%%%%%%%%%%%%%%%%%%%%%%%%%%%%%%%%%%%%%%%%%%%%%%%%%%%%%
%%%%%%%%%%%%%%%%%%%%%%%%%%%%%%%%%%%%%%%%%%%%%%%%%%%%%%%%%%%%%%%%%%%%%%%%%%%%%%%%%%%%%%%%%

%%%
%%%
%%%
%------------------------------
%------------------------------
%------------------------------

\small

\noindent
Benson Farb\\
Dept. of Mathematics, University of Chicago\\
5734 University Ave.\\
Chicago, IL 60637\\
farb@math.uchicago.edu
\medskip

\noindent
Christopher J. Leininger\\
Dept. of Mathematics, University of Illinois at Urbana-Champaign \\
273 Altgeld Hall, 1409 W. Green St. \\
Urbana, IL 61802\\
clein@math.uiuc.edu
\medskip

\noindent
Dan Margalit\\
Dept. of Mathematics, Tufts University\\
503 Boston Ave \\
Medford, MA 02155\\
dan.margalit@tufts.edu

\end{document}